\def\draft{n}
\documentclass[12pt]{amsart}
\usepackage[headings]{fullpage}
\usepackage{amssymb,epic,eepic,epsfig}
\usepackage{epstopdf}
\usepackage{graphicx}
\usepackage{texdraw}
\usepackage{url}
\usepackage[bookmarks=true,%
    colorlinks=true,%
    linkcolor=blue,%
    citecolor=blue,%
    filecolor=blue,%
    menucolor=blue,%
    urlcolor=blue,%
    breaklinks=true]{hyperref}
\usepackage{float}

\newtheorem{theorem}{Theorem}[section]
\theoremstyle{definition}
\newtheorem{proposition}[theorem]{Proposition}
\newtheorem{lemma}[theorem]{Lemma}
\newtheorem{definition}[theorem]{Definition}
\newtheorem{remark}[theorem]{Remark}
\newtheorem{corollary}[theorem]{Corollary}
\newtheorem{conjecture}[theorem]{Conjecture}

\newtheorem{example}[theorem]{Example}
\newtheorem{exercise}[theorem]{Exercise}

\def\printname#1{
        \if\draft y
                \smash{\makebox[0pt]{\hspace{-0.5in}
                        \raisebox{8pt}{\tt\tiny #1}}}
        \fi
}

\newcommand{\psdraw}[2]
         {\begin{array}{c} \hspace{-1.3mm}
        \raisebox{-4pt}{\epsfig{figure=draws/#1.eps,width=#2}}
        \hspace{-1.9mm}\end{array}}

\newlength{\standardunitlength}
\catcode`\@=11
\long\def\@makecaption#1#2{%
     \vskip 10pt

\setbox\@tempboxa\hbox{
       \small\sf{\bfcaptionfont #1. }\ignorespaces #2}%
     \ifdim \wd\@tempboxa >\captionwidth {%
         \rightskip=\@captionmargin\leftskip=\@captionmargin
         \unhbox\@tempboxa\par}%
       \else
         \hbox to\hsize{\hfil\box\@tempboxa\hfil}%
     \fi}
\font\bfcaptionfont=cmssbx10 scaled \magstephalf
\newdimen\@captionmargin\@captionmargin=2\parindent
\newdimen\captionwidth\captionwidth=\hsize
\catcode`\@=12

\def\lbl#1{\label{#1}\printname{#1}}

\def\BN{\mathbb N}
\def\BZ{\mathbb Z}

\def\BQ{\mathbb Q}

\def\W{W}

\def\calP{\mathcal P}

\def\calE{\mathcal{E}}

\def\a{\alpha}

\def\La{\Lambda}
\def\l{\lambda}

\def\s{\sigma}

\def\e{\epsilon}

\def\d{\delta}
\def\b{\beta}
\def\th{\theta}

\def\s{\sigma}

\def\calL{\mathcal L}
\def\longto{\longrightarrow}

\def\calL{\mathcal{L}}

\def\fg{\mathfrak{g}}
\def\fsl{\mathfrak{sl}}

\def\be{  \begin{equation} }
\def\ee{  \end{equation} }
\def\coeff{\mathrm{coeff}}

\begin{document}

\title[A stability conjecture for the colored Jones polynomial]{
A stability conjecture for the colored Jones polynomial}
\author{Stavros Garoufalidis}
\address{School of Mathematics \\
         Georgia Institute of Technology \\
         Atlanta, GA 30332-0160, USA \newline
         {\tt \url{http://www.math.gatech.edu/~stavros }}}
\email{stavros@math.gatech.edu}
\author{Thao Vuong}
\address{School of Mathematics \\
         Georgia Institute of Technology \\
         Atlanta, GA 30332-0160, USA \newline
         {\tt \url{http://www.math.gatech.edu/~tvuong }}}
\email{tvuong@math.gatech.edu}
\thanks{SG was supported in part by NSF grant DMS-11-05678. 
\newline
1991 {\em Mathematics Classification.} Primary 57N10. Secondary 57M25.
}

\date{October 27, 2013}


\begin{abstract}
We formulate a stability conjecture for the coefficients of the colored 
Jones polynomial of a knot, colored by irreducible representations in a 
fixed ray of a simple Lie algebra, and verify it for all torus knots and 
all simple Lie algebras of rank $2$.
Our conjecture is motivated by a structure theorem for the degree and
the coefficients of a $q$-holonomic sequence of polynomials given in \cite{Ga2}
and by a stability theorem of the colored Jones polynomial of an 
alternating knot given in \cite{GL2}.
We illustrate our results with sample computations.
\end{abstract}

\maketitle

\tableofcontents



\section{Introduction}
\lbl{sec.intro}

\subsection{The degree and coefficients of a $q$-holonomic sequence}
\lbl{sub.LT0}

Our goal is to formulate a stability conjecture for the coefficients of 
$q$-holonomic sequences that appear naturally in Quantum Knot Theory \cite{GL}.
Our conjecture is motivated by 
\begin{itemize}
\item[(a)]
a structure theorem for the degree and coefficients of a $q$-holonomic
sequence of polynomials given in \cite{Ga2},
\item[(b)]
a stability theorem of the colored Jones polynomial of an alternating knot 
\cite{GL2}.
\end{itemize}
To discuss our first motivation, 
recall \cite{Z} that a sequence $(f_n(q))$ is $q$-{\em holonomic} 
if it satisfies a linear recursion 
$$
\sum_{j=0}^d c_j(q^n,q) f_{n+j}(q)=0
$$
for all $n$ where $c_j(u,v) \in \BZ[u,v]$ and $c_d \neq 0$. Here,
$f_n(q)$ is either in $\BZ[q^{\pm 1}]$, the ring of Laurent polynomials with
integer coefficients, or more generally in $\BQ(q)$, the field of
rational functions with rational coefficients or even $\BZ((q))$, the ring
of Laurent power series in $q$ $\sum_{j \in \BZ} a_j q^j$ (with $a_j$ integers,
vanishing when $j$ is small enough). $\BZ((q))$ has a subring $\BZ[[q]]$
of formal power series in $q$, where $a_j=0$ for $j <0$.
The {\em degree} $\d^*(f(q))$ of $f(q) \in \BZ((q))$ is the smallest integer 
$m$ such that $q^m f(q) \in \BZ[[q]]$.

Thus, we can expand every non-zero sequence $(f_n(q))$ in the form
\begin{equation}
\lbl{eq.fna}
f_n(q)= a_0(n) q^{\d^*(n)} + a_1(n) q^{\d^*(n)+1 }
+ a_2(n) q^{\d^*(n)+2 } + \dots
\end{equation}
where $\d^*(n)$ is the degree of $f_n(q)$ and $a_k(n)$ is the $k$-th 
coefficient of $q^{-\d^*(n)} f_n(q)$, reading from the left. 
We will often call $a_k(n)$ the $k$-th {\em stable coefficient} of the 
sequence $(f_n(q))$.

In \cite{Ga2} it was proven that if $(f_n(q))$ is $q$-holonomic, then
\begin{itemize}
\item
$\d^*(n)$ is a {\em quadratic quasi-polynomial} for all but finitely
many values of $n$,
\item
for every $k \in \BN$, $a_k(n)$ is {\em recurrent} for all but finitely many 
values of $n$.
\end{itemize}
Recall that a quasi-polynomial (of degree at most $d$) is a function of 
the form
$$
p: \BN \longto \BZ, \qquad n \mapsto p(n)=\sum_{j=0}^d c_j(n) n^j
$$
where $c_j: \BN \longto \BQ$ are periodic functions. Let $\calP$ denote
the ring of integer-valued quasi-polynomials. A recurrent sequence is one that
satisfies a linear recursion with constant coefficients. Recurrent sequences
are well-known in number theory under the name of {\em Generalized 
Exponential Sums} \cite{vP,EvPSW}. The latter are expressions of the form
$$
a(n)=\sum_{i=1}^m A_i(n) \a_i^n
$$
with {\em roots} $\a_i$, $1 \leq i \leq m$ distinct algebraic numbers and 
polynomials $A_i$. Integer-valued generalized exponential sums form a ring 
$\calE$, which contains a subring $\calP$ that consists of integer-valued
exponential sums whose roots are complex roots of unity.

\subsection{Stability of the colored Jones polynomial of an alternating 
link}
\lbl{sub.stability}

The second motivation of our Conjecture \ref{conj.1} below comes from
the stability theorem of \cite{GL2} that concerns the colored Jones
polynomial of an alternating link. Recall the notion of convergence in
the completed ring $\BZ((q))=\varprojlim_n \BZ[q^{\pm 1}]/(q^n)$. 
Given $f_n(q),f(q) \in \BZ((q))$, we write that
$$
\lim_{n\to \infty} f_n(q) = f(q)
$$
if there exists $C$ such that $\d^*(f_n(q)) > C$ for all $n$, and 
for every $m \in \BN$ there exists $N_m \in \BN$ such that
$$
f_n(q)-f(q) \in q^m \BZ[[q]]\,.
$$
The next definition of stability appears in \cite{GL} and the notion
of its tail is inspired by Dasbach-Lin \cite{DL}.

\begin{definition}
We say that a sequence $f_n(q) \in \BZ[[q]]$ is {\em stable} 
if there exists a series $F(x,q)=\sum_{k=0}^\infty \Phi_k(q) x^k \in 
\BZ((q))[[x]]$ such that for every $k \in \BN$, we have
\be
\lbl{eq.kstable}
\lim_{n \to \infty} \,\, q^{-k(n+1)}
\left(f_n(q) -\sum_{j=0}^k \Phi_j(q) q^{j(n+1)}\right)=0 \,.
\ee
We will call $F(x,q)$ the $(x,q)$-{\em tail} (in short, the tail)
of the sequence $(f_n(q))$.
\end{definition}
Examples of stable sequences are the shifted colored Jones polynomials
of an alternating link. Let $J_{K,n}(q)\in \BZ[q^{\pm 1/2}]$ denote the colored 
Jones polynomial of a link $K$ colored by the $(n+1)$-dimensional irreducible 
representation of $\fsl_2$ (see \cite{Tu,Tu2}). Let $\d^*_K(n)$ and 
$a_{K,0}(n) $ denote the degree and the $0$-th stable coefficient 
of $J_{K,n}(q)$. It is well-known known that $a_{K,0}(n)=(-1)^{c_- n}$ where
$c_-$ is the number of negative crossings of $K$ \cite{Li}. 

\begin{theorem}
\lbl{thm.GL2}\cite{GL2}
If $K$ is an alternating link, then the sequence 
$a_{K,0}(-n) q^{-\d^*_K(n)} J_{K,n}(q) \in \BZ[q]$ is stable.
\end{theorem} 

\subsection{$c$-stability}
\lbl{sub.cstability}

We are now ready to introduce the notion of $c$-stability.

\begin{definition}
\lbl{def.stab}
We say that a sequence $f_n(q) \in \BZ((q))$ with $q$-degree
$\d^*(n)$ is {\em $c$-stable} (i.e., {\em cyclotomically stable})
if there exists a series $F(n,x,q)=\sum_{k=0}^\infty \Phi_k(n,q) x^k \in 
\calP((q))[[x]]$ such that for every $k \in \BN$, we have
\be
\lbl{eq.kpolystable}
\lim_{n \to \infty} \,\, q^{-k(n+1)}
\left(q^{-\d^*(n)}f_n(q) -\sum_{j=0}^k \Phi_j(n,q) q^{j(n+1)}\right)=0 \,.
\ee
We will call $F(n,x,q)$ the $(n,x,q)$-tail (in short, tail) 
of the sequence $(f_n(q))$.
\end{definition}

\begin{remark}
\lbl{rem.cstable.coeff}
The stable coefficients of a $c$-stable sequence $(f_n(q))$ are 
quasi-polynomials. I.e., with the notation of Equation \eqref{eq.fna},
we have that $a_k \in \calP$ for all $k$. In fact, if $(f_n(q))$ is $c$-stable
and $l \in \BN$, the stable coefficients of the sequence 
$$
q^{-l(n+1)}
\left(f_n(q) -\sum_{j=0}^{l-1} \Phi_j(q) q^{j(n+1)}\right)
$$
are quasi-polynomials.
\end{remark}

\subsection{Our results}
\lbl{sub.results}

For a knot $K$ in $S^3$, colored by an irreducible representation $V_{\l}$
of a simple Lie algebra $\fg$ with highest weight $\l$, one can define
the colored Jones polynomial $J^{\fg}_{K,V_{\l}}(q) \in \BZ[q^{\pm 1}]$
\cite{Tu,Tu2}. This requires a rescalled definition of $q$, which depends
only on the Lie algebra and not on the knot, and is discussed carefully
in \cite{Le}. In \cite{GL} it was shown that for every knot $K$
and every simple Lie algebra $\fg$, the function $\l \mapsto
J^{\fg}_{K,V_{\l}}(q)$ (and consequently the sequence $(J^{\fg}_{K,n\l}(q))$)
is $q$-holonomic.

\begin{conjecture}
\lbl{conj.1}
Fix a knot $K$, a simple Lie algebra $\fg$ and a dominant weight $\l$
of $\fg$. Then the sequence $(J^{\fg}_{K,n\l}(q))$ of colored Jones polynomials
is $c$-stable.
\end{conjecture}

\begin{theorem}
\lbl{thm.1}
Conjecture \ref{conj.1} holds for all torus knots and all rank $2$ simple 
Lie algebras.
\end{theorem}
For a precise statement and for a computation of the tail, see Theorem
\ref{thm.stable.limit}.

\begin{remark}
\lbl{rem.alt}
Theorem \ref{thm.GL2} implies that if $K$ is an alternating knot with
$c_-$ crossings and $k \in \BN$, the $k$-th stable coefficient $a_{K,k}(n)$ of 
the sequence $(J_{K,n}(q))$ is given by 
$$
a_{K,k}(n)=(-1)^{c_- n} \coeff(\Phi_{K,0}(q),q^k)
$$
and satisfies the first order linear recurrence relation
$$
a_{K,k}(n+1)-(-1)^{c_-}a_{K,k}(n)=0.
$$
Here $\coeff(f(q),q^k)$ denotes the coefficient of $q^k$ in $f(q) \in 
\BZ((q))$. The stable coefficients $c_{K,k}$ of an alternating knot $K$
are studied in \cite{GV2,GV3}. 
In all examples of the colored Jones polynomial of a knot that have been 
analyzed (this includes alternating knots, torus knots and the 2-fusion
knots), the $k$-stable coefficient is a quasi-polynomial of degree $0$, 
i.e., it is constant on suitable arithmetic progressions. One might think 
that this holds for all simple Lie algebras. Example \ref{ex.T45} below
shows that this is not the case, hence the notion of $c$-stability is 
necessary.
\end{remark}

\subsection{A sample of $q$-series}
\lbl{sub.sample}

In this section we give a concrete sample of tails and $q$-series that
appear in our study.

\begin{example}
\lbl{ex.T23}

Consider the theta series given by \cite{123}
\begin{equation}
\lbl{eq.theta}
\theta_{b,c}(q) =\sum\limits_{s\in\mathbb{Z}} (-1)^s q^{\frac{b}{2}s^2+cs}
\end{equation}
In Section \ref{sec.examples} we will prove the following.

\begin{theorem}
\lbl{thm.T23}
The tail of the $c$-stable sequence $(J^{\mathfrak{g}}_{T(2,b),n\lambda_1}(q))$ 
for $b>2$ odd is given by
\begin{align*}
\frac{\theta_{b,\frac{b}{2}-1}(q)(1+q^3x^2)
+q^3\theta_{b,\frac{b}{2}+2}(q)x }{(1-q)(1-qx)(1-q^2x)} 
\end{align*}
In particular, for the trefoil, i.e., $b=3$ the tail equals to
$$
(q)_\infty\frac{1-qx+q^3x^2}{(1-q)(1-qx)(1-q^2x)} 
$$
\end{theorem}
\end{example}

\begin{example}
\lbl{ex.T45}
The tail of the $c$-stable sequence $(J^{A_2}_{T(4,5),n\rho}(q))$ is given by  
\begin{align*}
\frac{1}{(1-xq)^2(1-x^2q^2)}
\left( A_0(q) + n A_1(q) \right)
\end{align*}
where $A_0(q), A_1(q) \in \BZ[[q]]$ are given explicitly in Proposition
\ref{prop.A2.T45}. The first few terms of those $q$-series are given by
\begin{align*}
A_0(q) &=
1-2 q+2 q^3-q^4+q^{48}-2 q^{55}-2 q^{57}+2 q^{63}+2 q^{66}+2 q^{69}+2 q^{75}
-q^{76}-2 q^{78}-2 q^{81}
\\ & -2 q^{82}-q^{84}-2 q^{85}+2 q^{87} + \dots
\\
A_1(q) &=
1-2 q+2 q^3-q^4-q^6+2 q^9+2 q^{10}-2 q^{12}-4 q^{15}-q^{18}+2 q^{19}+2 q^{21}
+3 q^{22}-2 q^{27}
\\ & -2 q^{30}-2 q^{33}+4 q^{36}-q^{42}-2 q^{46}+q^{48}-2 q^{49}
+2 q^{51}+2 q^{55}+4 q^{57}-2 q^{58}+2 q^{60}-4 q^{64}
\\ & -2 q^{66}-2 q^{69}
-2 q^{73}-q^{76}+4 q^{78}+2 q^{81}+2 q^{82}+q^{84}+2 q^{85}-2 q^{87} + \dots
\end{align*}
It follows that for every fixed $k$, the $k$-th stable coefficient $a_k(n)$
of $(J^{A_2}_{T(4,5),n\rho}(q))$ satisfies the linear recursion relation
$$
a_k(n+2)-2 a_k(n+1)+ a_k(n)=0
$$
for all $n$. 

We leave as an exercise to the reader to show that
\begin{align*}
A_1(q) &=
\sum_{m_1,m_2 \in \BZ} q^{20(m_1^2 + 3 m_1 m_2 + 3 m_2^2)
+ 2 m_1 + 3 m_2}
(1 - q^{4 m_1 +1}) (1 - q^{4 m_1 + 12 m_2 +1}) (1 - q^{8 m_1 + 12 m_2+2})
\\
&=(q)_\infty\left( \sum_{n \in \BZ} (-1)^n q^{\frac{15n^2+n}{2}} - 
\sum_{n \in \frac{3}{5}+\BZ} (-1)^n q^{\frac{15n^2+n}{2}} \right) \\
&=(q)_\infty\left(
(q^7,q^{15})_\infty (q^8,q^{15})_\infty (q^{15},q^{15})_\infty-q (1-q^2) 
(q^{13},q^{15})_\infty (q^{15},q^{15})_\infty (q^{17},q^{15})_\infty \right)
\end{align*}
where as usual, $(x,q)_\infty=\prod_{k=0}^\infty (1-q^k x)$ and 
$(q)_\infty=(q,q)_\infty$.
The book \cite{123} is an excellent source for proving such identities.
\end{example}


\section{The colored Jones polynomial of a torus knot}
\lbl{sec.cjtorus}

\subsection{The Jones-Rosso formula}
\lbl{sub.JR}

To verify Conjecture \ref{conj.1} for all torus knots $T(a,b)$ (where
$0<a<b$ and $a$ and $b$ are coprime integers), we will
use the formula of Jones-Rosso \cite{JR}. It states that

\begin{equation}
\lbl{eq.JR}
J^{\fg}_{T(a,b),\l}(q)=\frac{\th_{\l}^{-ab}}{d_{\l}}
\sum_{\mu \in S_{\l,a}} m^{\mu}_{\l,a} d_{\mu} \th_{\mu}^{\frac{b}{a}}
\end{equation}
where 
\begin{itemize}
\item
$d_{\lambda}$ is the {\em quantum dimension} of $V_{\lambda}$ and $\th_{\l}$ is the 
eigenvalue of the {\em twist} operator on the representation $V_{\l}$ given by:
\begin{equation}
d_{\l}=\prod_{\a>0} \frac{[(\l+\rho,\a)]}{[(\rho,\a)]},
\qquad \th_{\l}=q^{\frac{1}{2}(\l,\l+2\rho)}, 
\qquad [n]=\frac{q^{\frac{n}{2}}-q^{-\frac{n}{2}}}{q^{\frac{1}{2}}-q^{-\frac{1}{2}}} \,.
\end{equation}
\item
$m^{\mu}_{\l,a} \in \BZ$ is the multiplicity of $V_{\mu}$ in the {\em
$a$-plethysm} of $V_\l$ where where $\psi_a$ denote the $a$-Adams operation. 
I.e., we have:
\begin{equation}
\lbl{eq.psila}
\psi_a(ch_{\l})=\sum_{\mu \in S_{\l,a}} m^{\mu}_{\l,a} ch_{\mu}
\end{equation}
where $ch_\l$ is the formal character of $V_\l$. 
\end{itemize}

To describe the plethysm multiplicity $m^{\mu}_{\l,a}$ and the summation set
$S_{\l,a}$, recall the
Kostant multiplicity formula \cite{Ko1} which expresses the multiplicities 
$m_{\l}^{\mu}$ of the $\mu$-weight space of $V_{\l}$ in terms of the 
Kostant partition function $p$:
\begin{equation}
\lbl{eq.komult}
m_\l^\mu=\sum\limits_{\sigma\in\W}(-1)^\sigma p(\sigma(\l+\rho)-\mu-\rho)
\end{equation}
As usual, $W$ is the Weyl group of the simple Lie algebra $\fg$ and $\rho$
is half the sum of its positive roots.

\begin{lemma}
\lbl{lem.mula}
\rm{(a)} We have:
\begin{equation}
\lbl{eq.mula}
m^{\mu}_{\l,a}=\sum_{\sigma\in \W}(-1)^\sigma m^{\frac{\mu+\rho-\sigma(\rho)}{a}}_{\l}
\end{equation}
where the summation is over the elements $\sigma \in \W$ 
such that $\frac{\mu+\rho-\sigma(\rho)}{a}$  is 
in the weight lattice (but not necessarily a dominant weight). 
\newline
\rm{(b)} It follows that
\be
\lbl{eq.Sla}
S_{\l,a}= \left[ \bigcup\limits_{\s \in \W} \left( \s(\rho)-\rho+ a 
\Pi_{\l} \right) \right] \cap \La^+
\ee
where $\Pi_\l$ is the set of all weights of $V_\l$.
\end{lemma}

\begin{remark}
The Jones-Rosso formula \eqref{eq.JR} combined with Equations 
\eqref{eq.komult} and \eqref{eq.mula} imply that that we can write
\be
\lbl{eq.fWW}
J^{\fg}_{T(a,b),\l}(q)=\sum_{\s,\s' \in W} J^{\fg}_{T(a,b),\l,\s,\s'}(q)
\ee
for some rational functions $J^{\fg}_{T(a,b),\l,\s,\s'}(q)$. It is easy to see 
that the sequences $(J^{\fg}_{T(a,b),n\l,\s,\s'}(q))$ are $q$-holonomic (with
respect to $n$) and $c$-stable. If cancellation of the leading and trailing
terms did not occur in Equation \eqref{eq.fWW}, it would imply a short proof
of Theorem \ref{thm.1} for all torus knots and all simple Lie algebras.
Unfortunately, after we perform the sum in Equation \eqref{eq.fWW} 
cancellation occurs and the degree of the summand is much lower than the 
degree of the sum. This already happens for $A_2$ and the trefoil, an 
alternating knot. This cancellation is responsible for the minimizer 
$\mu_{\l,a}$ to be of order $O(\l)$ rather than $O(1)$ in case $A_2$, 
part (b) of Theorem \ref{thm.min}. 
\end{remark}

\subsection{The degree of the colored Jones polynomial}
\lbl{sub.shifted}

The Jones-Rosso formula can be written in the form
\begin{align}
\lbl{formula}
J^{\fg}_{T(a,b),\l}(q) 
&= \frac{q^{-\frac{ab}{2}(\l,\l)-(-1+ab)(\l,\rho)}}{\prod
\limits_{\a\succ 0}(1-q^{(\l+\rho,\a)})}
\sum\limits_{\mu\in S_{\l,a}}q^{\frac{b}{2a}(\mu,\mu)+(-1+\frac{b}{a})(\mu,\rho)}
\prod\limits_{\a\succ 0}(1-q^{(\mu+\rho,\a)}) \,.
\end{align}
When the dominant weight $\lambda$ and the torus knot $T(a,b)$ is fixed, 
the minimum the and maximum degree of the summand are positive-definite 
quadratic forms $f^*(\mu)$ and $f(\mu)$ given by
\begin{subequations}
\begin{align}
\lbl{min.deg}
f^*(\mu) &=\frac{b}{2a}(\mu,\mu)+\left(-1+\frac{b}{a}\right)(\mu,\rho)
-\frac{ab}{2}(\l,\l)-(-1+ab)(\l,\rho)\,
\\
\lbl{max.deg}
f(\mu)&=\frac{b}{2a}(\mu,\mu)+\left(1+\frac{b}{a}\right)(\mu,\rho)
-\frac{ab}{2}(\l,\l)-(1+ab)(\l,\rho)\,
\end{align}
\end{subequations}

In Section \ref{sec.thm.max} we will prove the following.

\begin{theorem}
\lbl{thm.max}
Fix a simple Lie algebra $\fg$ and a torus knot $T(a,b)$. The quadratic form 
$f(\mu)$ achieves maximum uniquely at $M_{\l,a}=a \l\in S_{a,\l}$. 
Moreover, $m^{M_{\l,a}}_{\l,a}=1$.
\end{theorem}

The next theorem states that $f^*(\mu)$ has a unique minimizer which we 
denote by $\mu_{\l,a}$  and describes $\mu_{\l,a}$ explicitly for all simple
Lie algebras of rank $2$. Below, $\{\l_1,\l_2\}$ are the dominant weights
of a simple Lie elgebra of rank $2$. Its proof
is given in Section \ref{sec.thm.min} using a case-by-case analysis.

\begin{theorem}
\lbl{thm.min}
When $\fg$ is a simple Lie algebra of rank $2$, then
\newline
\rm{(a)}The quadratic form $f^*(\mu)$ achieves minimum uniquely at 
$\mu_{\l,a}\in S_{a,\l}$ and  $m^{\mu}_{\l,a} \neq 0$.
\newline
\rm{(b)} For a dominant weight $\l=m_1\l_1+m_2\l_2$, we have 
\newline
For $A_2$: 
$$
\mu_{\l,2}= \begin{cases}
(m_1-m_2)\l_2 & \text{if} \,\, m_1 \geq m_2 \\
(m_2-m_1)\l_1 & \text{if}  \,\, m_1 \leq m_2
\end{cases}
\qquad
\mu_{\l,3}=0
$$
and 
$$
\mu_{\l,a}= \begin{cases}
0 & \text{if} \,\, m_1 \equiv m_2 \bmod 3\\
(a-3)\l_1 & \text{if}  \,\, m_1 \equiv m_2 +1 \bmod 3 \\
(a-3)\l_2 & \text{if}  \,\, m_1 \equiv m_2 +2 \bmod 3 
\end{cases}
\quad
\text{for $a \geq 4$}
$$
For $B_2$:
$$
\mu_{\l,2}=\begin{cases}
\l_1 & \text{if} \,\, m_1=0 \,,  m_2\equiv 1\bmod 2\\
0 &\text{otherwise} 
\end{cases}
\qquad
\mu_{\l,3}= \begin{cases}
0 & \text{if} \,\, m_1,m_2 \equiv 0 \bmod 2 \\
2\lambda_2 & \text{if} \,\, m_1 \equiv 1 \bmod 2\,, m_2 \equiv 0 \bmod 2 \\
\lambda_1+\lambda_2 &\text{if} \,\, m_2 \equiv 1 \bmod 2 
\end{cases}
$$
$$
\mu_{\l,4}=0
\qquad
\mu_{\l,a}= \begin{cases}
0 & \text{if} \,\, m_2 \equiv 0 \bmod 2 \\
(a-4)\lambda_2 & \text{if} \,\,m_2 \equiv 1 \bmod 2 
\end{cases}
\quad
\text{for $a\geq 5$}
$$
For $G_2$:
$$
\mu_{\l,a}=0 \,\, \text{for all $a\geq 2$}
$$
\end{theorem}
Theorem \ref{thm.min} part (b) implies the following.
\begin{corollary}
$\mu_{n\l,a}$ is a piecewise quasi-linear function of $n$ for 
$n \gg 0$. 
\end{corollary}
Let $\d^*_{K}(\l)$ and $\d_{K}(\l)$ denote the {\em minimum} and the 
{\em maximum} degree of the colored Jones polynomial $J^{\fg}_{K,V_{\l}}(q)$
with respect to $q$.

\begin{corollary}
\lbl{cor.minmax}
We have:
\begin{subequations}
\begin{align}
\lbl{eq.mindeg}
\d^*_{T(a,b)}(\l)&= f^*(\mu_{\l,a}) \\
\lbl{eq.maxdeg}
\d_{T(a,b)}(\l)&= f(a\l)\,.
\end{align}
\end{subequations}
\end{corollary}


\section{Some lemmas about stability}
\lbl{sec.somelemmas}

In this section we collect some lemmas about stable sequences.

\begin{lemma}
\lbl{stability.lemma1}
Fix natural numbers $c$ and $d$ and consider $g_n(q)=\frac{f_n(q)}{1-q^{cn+d}}$.
Then $(f_n(q))$ is stable if and only if $(g_n(q))$  is stable. In that
case, their corresponding tails $F(x,q)$ and $G(x,q)$ satisfy 
\be
\lbl{eq.FG}
G(x,q)=\frac{F(x,q)}{1-q^dx^c} \,.
\ee
\end{lemma}

\begin{proof}
Let 
$$
F(x,q)=\sum\limits_{k=0}^\infty\phi_k(q)x^k, \qquad
G(x,q)=\sum\limits_{k=0}^\infty\psi_k(q)x^k \,.
$$
If $F$ and $G$ satisfy Equation \eqref{eq.FG}, collecting powers
of $x^k$ on both sides implies that
\begin{equation}
\lbl{lemma.stable1}
\psi_k(q)=\sum\limits_{i+jc=k}\phi_i(q)q^{jd} \,.
\end{equation}
Assume that $f_n(q)$ is stable, and define $\psi_k(q)$ by Equation
\eqref{lemma.stable1}. We will prove by induction on $k$ that $g_n(q)$
is $k$-stable with corresponding limit $\psi_k(q)$.

Let 
\begin{align*}
\alpha_{0,n}(q)=& f_n(q)\\
\alpha_{k,n}(q)=&q^{-n}(\alpha_{k-1,n}-\phi_{k-1}(q))\\
=& q^{-kn}\left(f_n(q)-\sum\limits_{l=0}^{k-1}\phi_l(q)q^{ln}\right),~~k\geq 1
\end{align*}
and
\begin{align*}
\beta_{0,n}(q)=& g_n(q)\\
\beta_{k,n}(q)=&q^{-n}(\beta_{k-1,n}-\psi_{k-1}(q))\\
=& q^{-kn}\left(g_n(q)-\sum\limits_{l=0}^{k-1}\psi_l(q)q^{ln}\right),~~k\geq 1
\end{align*}

For $k=0$, the $0-$limit of $g_n(q)$ is 
$\lim_{n\to \infty}g_n(q)=\lim_{n\to\infty} \frac{f_n(q)}{1-q^{cn+d}}
=\phi_0(q)=\psi_0(q)$.

Assuming by induction that $g_n(q)$ is $(k-1)$-stable, we have
\begin{align*}
\beta_{k,n}(q) =& q^{-kn}\left(\frac{f_n(q)}{1-q^{cn+d}}
-\sum\limits_{l=0}^{k-1}\sum\limits_{i+jc=l}\phi_i(q)q^{jd}q^{(i+jc)n}\right)\\
=&  q^{-kn}\left(f_n(q)\sum\limits_{j=0}^\infty q^{j(cn+d)}
-\sum\limits_{0\leq i+jc\leq k-1}\phi_i(q)q^{in}q^{j(cn+d)}\right)\\
=&q^{-kn}\sum\limits_{j=0}^{\lfloor\frac{k-1}{c}\rfloor}q^{j(cn+d)}\left(f_n(q)
-\sum\limits_{i=0}^{k-1-jc}\phi_i(q)q^{in}\right)+
q^{-kn}\sum\limits_{j>\lfloor\frac{k-1}{c}\rfloor}q^{j(cn+d)}f_n(q)\\
=&\sum\limits_{j=0}^{\lfloor\frac{k-1}{c}\rfloor} q^{jd}q^{-(k-jc)n}\left(f_n(q)
-\sum\limits_{i=0}^{k-1-jc}\phi_i(q)q^{in}\right)+
q^{-kn}\sum\limits_{j>\lfloor\frac{k-1}{c}\rfloor}q^{j(cn+d)}f_n(q)\\
=& \sum\limits_{j=0}^{\lfloor\frac{k-1}{c}\rfloor} q^{jd}\alpha_{k-jc,n}(q)+q
^{-kn}\sum\limits_{j>\lfloor\frac{k-1}{c}\rfloor}q^{j(cn+d)}f_n(q)\\
=&\sum\limits_{j=0}^{\lfloor\frac{k-1}{c}\rfloor} q^{jd}\alpha_{k-jc,n}(q)+
q^{-kn}\sum\limits_{\lfloor\frac{k-1}{c}\rfloor <j\leq \frac{k}{c}}q^{j(cn+d)}f_n(q)+
q^{-kn}\sum\limits_{j > \frac{k}{c}}q^{j(cn+d)}f_n(q)\\
=&\sum\limits_{i+jc=k}q^{jd}\alpha_{i,n}(q)
+\sum\limits_{j > \frac{k}{c}}q^{n(jc-k)+jd}f_n(q)
\end{align*}
Therefore, 
$$
\lim_{n\to\infty}\beta_{k,n}(q)=\sum\limits_{i+jc=k}q^{jd}\phi_i(q)=\psi_k(q) 
\,.
$$
Conversely, if $(g_n(q))$ is stable, so is $(f_n(q))$.
\end{proof}

\begin{lemma}
\lbl{stability.lemma2}
Fix a rational polytope $P \subset [0,\infty)^r$ that intersects the interior
of every positive coordinate ray and a positive 
definite quadratic function $Q: \BZ^r \longto \BZ$. Let
$c:\BN\times\BZ^r \longto \BZ$ be such that for each fixed $v\in\BZ^r$ and
for $n \gg 0$, $c(n,v)=t(n,v)$ where $n \mapsto t(n,v)$ is a quasi-polynomial. 
For each natural number $n$ define
$$
T_n(q)=\sum\limits_{v\in nP\cap\calL}c(n,v)q^{Q(v)} \,.
$$
Then $(T_n(q))$ is $c$-stable and its $(n,x,q)$-tail is given by
$$
F(n,x,q)=\sum\limits_{v\in\calL\cap\mathbb{R}^r_+}t(n,v)q^{Q(v)} \,.
$$
\end{lemma}

\begin{proof}
Let $\phi_0(n,q)=\sum\limits_{v\in\calL\cap\mathbb{R}^r_+}t(n,v)q^{Q(v)}$.
We need to prove that for all $k \geq 0$, we have
$$
\lim_{n\to\infty}q^{-kn}(T_n(q)-\phi_0(n,q))=0 \,.
$$
We have 
\begin{align}
\lbl{eq.stab.lemma2}
q^{-kn}(T_n(q)-\phi_0(n,q))=&q^{-kn}\sum\limits_{v\in P_n}(c(n,v)-t(n,v))q^{Q(v)}
-\sum\limits_{v\in (\calL\cap\mathbb{R}^r_+)\setminus P_n}t(n,v)q^{Q(v)-kn}\\
=& -\sum\limits_{v\in (\calL\cap\mathbb{R}^r_+)\setminus P_n}t(n,v)q^{Q(v)-kn}
\end{align}
for $n$ large enough.
Let us first assume that $Q$ is a quadratic form and let $d$ be the minimum of 
$Q$ on $\mathbb{R}^r\setminus P^{\circ}$. We will
prove that $d>0$. Indeed, since $Q$ is a positive definite form we only need to
minimize $Q$ over the union $F$ of the faces of $P$ that are not in the 
coordinate planes. Since $F$ is compact, $Q$ attains its minimum at some 
$v_0\in F$ and $d=Q(v_0)>0$ since $v_0\neq 0$. If 
$v\in \mathbb{R}^r\setminus nP^{\circ}$ (where $P^{\circ}$ is the interior of 
$P$) then $v=n v', v'\in \mathbb{R}^r\setminus P^{\circ}$,
so $Q(v)=Q(nv')=n^2Q(v')\geq dn^2$. Therefore the limit of the right hand 
side of Equation \eqref{eq.stab.lemma2} as $n$ approaches
infinity is zero.

If $Q$ is not a quadratic form we can write $Q=Q_2+Q_1$ where $Q_2$ is the 
quadratic part of $Q$. Then if $v\in \mathbb{R}^r\setminus nP^{\circ}$ we have
$Q(v)=Q_2(v)+Q_1(v)\geq dn^2+Q_1(v) > (d+1)n^2$ for $n$ large enough.
\end{proof}

\begin{remark}
\lbl{rm.lat}
Let $p\in P$. The \textit{tangent cone} Tan$(P,p)$ is defined to be the set
of all directions $v$ that one can go and stay in $P$: 
$$
\text{Tan}(P,p)=\{v\in \mathbb{R}^r | p+\e v\in P \, \text{for small}\, \e>0\}
$$
Lemma \ref{stability.lemma2} still holds if we replace $nP$ with $n(P-p)$ or 
$nP-p$ and   $\calL$ with a union of a finite number of translates of 
$\calL$. In this setting, the stable series is
$$
F(n,x,q)=\sum\limits_{v\in \text{Tan}(P,p)\cap\BZ^r}t(n,v)q^{Q(v)} \,.
$$
\end{remark}
\begin{remark}
\lbl{rm.cstbl}
Supose that $f_n(q)$ satisfies $\delta^*(f_n(q))\geq cn^2$ for some 
$c>0,n\geq 0$ then $g_n(q)$ is c-stable if $g_n(q)+f_n(q)$ is c-stable and 
they have the same tails.
\end{remark}


\section{Stability of the multiplicity} 
\lbl{sec.mult}

\subsection{Lie algebra notation}
\lbl{sub.notation}

Let us recall some standard notation from \cite{Bou,Hu}.
Let $\fg$ denote a simple Lie algebra of rank $r$ with weight lattice $\La$,
root lattice $\La_r$ and normalized inner product $(\cdot , \cdot)$ on
$\La$. Let $\W$ be its Weyl group and 
$\La^+$ the set of all the dominant weights with respect to a fixed Weyl 
chamber. Let $\a_i$ (resp., $\l_i$), $1\leq i\leq r$, be the set of 
simple roots (resp., fundamental weights) of $\mathfrak{g}$.The root 
lattice $\Lambda_r$ has the partial order given by $\b\prec\a$ 
if and only if $\a-\b=\sum\limits_{i=1}^r n_i\a_i$ where 
$n_i\in\BN$, $i=1,\dots,r$.

For a dominant weight $\l \in \La^+$, let $V_{\l}$ denote the 
corresponding irreducible representation $V_{\l}$ and let 
$\Pi_{\l}$ denote the set of all of the weights of $V_{\l}$.

The {\em Kostant partition function} $p(\a)$ of an element of the root lattice
$\a$ is the sum of all ways of writing $\a$ as a nonnegative integer linear
combination of positive roots \cite{Ko1}.

\subsection{A formula for the plethysm multiplicity}
\lbl{sub.mmula}

In this section we prove Lemma \ref{lem.mula}.

\begin{proof}(of Lemma \ref{lem.mula})
{\rm (a)} We have
\begin{equation}
\lbl{eq.char1}
\psi_a(ch_\l) = \psi_a(\sum\limits_{\mu\in\Pi_\l}m_\l^\mu e_\mu)
=\sum\limits_{\mu\in\Pi_\l}m_\l^\mu \psi_a(e_\mu)
= \sum\limits_{\mu\in\Pi_\l}m_\l^\mu e_{a\mu} \,.
\end{equation}
From Equations \eqref{eq.psila} and \eqref{eq.char1} we have
\begin{equation}
\lbl{eq.char2}
\sum\limits_{\mu}m_{\l,a}^\mu ch_\mu = \sum\limits_{\mu\in\Pi_\l}m_\l^\mu e_{a\mu}
\,.
\end{equation}

Let us define $\omega(\mu):=\sum\limits_{\s\in W}(-1)^\s e_{\s(\mu)}$ by 
for $\mu\in \Lambda^+$. The Weyl character formula states that \cite{Hu}:
$$
\omega(\rho) ch_\l=\omega(\l+\rho) \,.
$$
Multiplying both sides of Equation \eqref{eq.char2} with $\omega(\rho)$ 
and applying Weyl's formula we have
\begin{equation}
\lbl{eq.char3}
\sum\limits_{\mu}m_{\l,a}^\mu \omega(\mu+\rho)
=(\sum\limits_{\mu\in \Pi_\l}m_\l^\mu e_{a\mu}) \omega(\rho)
\end{equation}
Replacing $\omega(\mu+\rho)$ with $\sum\limits_{\s\in W}(-1)^\s e_{\s(\mu+\rho)}$ 
and $\omega(\rho)$ with $\sum\limits_{\s\in W}(-1)^\s e_{\s(\rho)}$ in Equation 
\eqref{eq.char3} we have
\begin{align}
\sum\limits_{\mu} \sum\limits_{\s\in W}(-1)^\s m_{\l,a}^\mu e_{\s(\mu+\rho)}
=&(\sum\limits_{\mu\in \Pi_\l}m_\l^\mu e_{a\mu}) 
(\sum\limits_{\s\in W}(-1)^\s e_{\s(\rho)})\\
\lbl{eq.char4}
=& \sum\limits_{\mu\in \Pi_\l}\sum\limits_{\s\in W}(-1)^\s 
m_\l^\mu e_{a\mu+\s(\rho)}
\end{align}
Setting $\s(\mu+\rho)=\nu+\rho$ on the left hand side of Equation 
\eqref{eq.char4} and $a\mu+\s(\rho)=\nu+\rho$ on right hand side we have
\begin{equation}
\lbl{eq.char5}
\sum\limits_{\nu} \sum\limits_{\s\in W}(-1)^\s m_{\l,a}^{\s^{-1}(\nu+\rho)-\rho} 
e_{\nu+\rho} = \sum\limits_{\nu}\sum\limits_{\s\in W}(-1)^\s 
m_\l^{\frac{\nu+\rho-\s(\rho)}{a}} e_{\nu+\rho}
\end{equation}
But we want $\s^{-1}(\nu+\rho)-\rho$ to be a dominant weight, which can 
happen only when $\s=1$. Therefore Equation \eqref{eq.char5} becomes
\begin{equation}
\lbl{eq.char6}
\sum\limits_{\nu} m_{\l,a}^{\nu} e_{\nu+\rho} = 
\sum\limits_{\nu}\sum\limits_{\s\in W}(-1)^\s 
m_\l^{\frac{\nu+\rho-\s(\rho)}{a}} e_{\nu+\rho}
\end{equation}
 Identifying the coefficients of $e_{\nu+\rho}$ on both sides of Equation 
\eqref{eq.char6} gives us the desired equality.
\newline
{\rm (b)} This follows from the fact that $m_\l^\nu \neq 0$ if and only if
$\nu\in \Pi_\l$. If $\nu=\frac{\mu+\rho-\s(\rho)}{a}$ this means that 
$\mu\in \s(\rho)-\rho+a\Pi_\l$.
\end{proof}

\subsection{Stability of the plethysm multiplicity}
\lbl{sub.stamult}

In this section we will prove that the coefficients $m^{\mu+n\nu}_{n\lambda,a}$ 
is a piecewise quasi-polynomial for $n\gg 0$ where $\lambda\in \Lambda^+$, 
$\mu,\nu\in\Lambda$. A piecewise quasi-polynomial function on a rational
vector space is a rational polyhedral fan together with a quasi-polynomial
function on each chamber of the fan.
Piecewise quasi-polynomials appear naturally as {\em vector partition 
functions} \cite{St}. The {\em Kostant partition function} of a simple Lie 
algebra $\mathfrak{g}$ is a vector partition function (see \cite{Ko1}),
hence a piecewise quasi-polynomial.

\begin{theorem} 
\lbl{stab.theorem}
Let $\lambda\in \Lambda^+$, $\mu,\nu\in\Lambda$, then $m_{n\lambda,a}^{\mu+n\nu}$ 
is a piecewise quasi-polynomial in $n$ for $n\gg 0$.
\end{theorem}

\begin{proof}
We have
\begin{equation}
\lbl{mult.eq}
m^{\mu+n\nu}_{n\l,a}=\sum_{\sigma\in \W}(-1)^\sigma 
m^{\frac{\mu+n\nu+\rho-\sigma(\rho)}{a}}_{n\l}
\end{equation}
and by Kostant's multiplicity formula in \cite{Ko1}, we have
\begin{eqnarray*}
\lbl{kostant.formula}
m^{\frac{\mu+n\nu+\rho-\sigma(\rho)}{a}}_{n\l} &=& \sum_{\sigma'\in \W}(-1)^{\sigma'} 
p\left(\sigma'(n\l+\rho)-(\frac{\mu+n\nu+\rho-\sigma(\rho)}{a}+\rho)\right)\\
&=&\sum_{\sigma'\in \W}(-1)^{\sigma'} p\left(n\sigma'(\l) 
-(\frac{\mu+n\nu+\rho-\sigma(\rho)}{a}+\rho-\sigma'(\rho))\right)\\
&=&\sum_{\sigma'\in \W}(-1)^{\sigma'} p\left(n(\sigma'(\l)-\frac{\nu}{a}) 
-(\frac{\mu+\rho-\sigma(\rho)}{a}+\rho-\sigma'(\rho))\right)\\
&=& \sum_{\sigma'\in \W}(-1)^{\sigma'} p\left(n\l'-\alpha'\right)
\end{eqnarray*}
Assume that $n\l'-\alpha'$ can be written as a linear combination of 
positive roots of $\mathfrak{g}$ so 
that $p(n\l'-\alpha')\neq 0$. For $n\gg 0$, $n\l'-\alpha'$ stays in some 
fixed Kostant chamber and it follows from 
Theorem 1 in \cite{St} that $p(n\l'-\alpha')$ is a quasi-polynomial in $n$. 
Since $m_{n\lambda,a}^{\mu+n\nu}$ is a finite sum of quasi-polynomials in $n$, 
it is also a quasi-polynomial in $n$.
\end{proof}


\section{The summation set}
\lbl{sec.Sla}

\subsection{A lattice point description of the summation set}
\lbl{sub.Sla}

In this section give a lattice point description of the summation set 
$S_{\l,a}$. Let $P_\l$ denote the convex polytope defined by the 
convex hull of $\Pi_\l \cap \La^+$.

\begin{lemma}
\lbl{lem.Plambda}
 For all $\l,a$ we have:
\begin{equation}
\lbl{eq.inclusion}
S_{\l,a}\subseteq \calL_{\l,a} \cap  P_{a\l}
\end{equation}
where
\begin{equation}
\lbl{trans.lat}
\calL_{\l,a}=\bigcup\limits_{\s\in W}\left( a\l+\s(\rho)-\rho+a\Lambda_r \right)
\end{equation}
is a finite union of translates of the lattice $a\Lambda_r$. 
Let 
\begin{equation}
\lbl{eq.missing}
R_{\l,a}=
\left(\calL_{\l,a} \cap  P_{a\l} \right)\setminus S_{\l,a}
\end{equation}
denote the set of {\em missing points}.
\end{lemma}

\begin{proof}
Recall that $P_\lambda$ consists of all $\alpha$ that satisfy (see \cite{Hu}),
\begin{align}
\lbl{ineq1}
(\a,\a_i) &\geq  0 \qquad (\l-\a,\l_i) &\geq 0 
\end{align}
for all $i=1,\dots,r$. 
We first prove that $S_{\l,a}\subseteq P_{a\l}$. By Lemma \ref{lem.mula}(b), 
we can write $\mu=a\nu+\sigma(\rho)-\rho\in\Lambda^+$ where $\nu\in\Pi_\l$. 
Since $\mu\in\Lambda^+$, Inequality \eqref{ineq1} holds trivially. To prove 
the second part of Inequality \eqref{ineq1}, it suffices to show that 
$(\mu,\l_i) \leq (a\l,\l_i)$ for every $1\leq i\leq r$. We have 
\begin{align*}
(a\l,\l_i)-(\mu,\l_i) & =  (a\l-\mu,\l_i)\\
& =  (a(\l-\mu)+\rho-\sigma(\rho),\l_i)\\
& \geq  0
\end{align*}
since $a(\l-\mu)+\rho-\sigma(\rho)$ is a $\BN$-linear combination
of positive roots. 

Let $\mu=a\nu+\s(\rho)-\rho\in S_{\l,a}$ where $\nu\in\Pi_\l$ and 
$\sigma\in\W$. Then $\mu=a(\l-\a)+\s(\rho)-\rho$
 where $\a$ is some positive root. It follows that $\mu\in a
\Lambda_r+a\l+\s(\rho)-\rho\subset \calL_{\l,a}$.
This proves that $S_{\l,a}\subseteq \calL_{\l,a}$ and
completes the proof of the lemma.
\end{proof}

\begin{remark}
\lbl{rem.notequality}
The inclusion in Equation \eqref{eq.inclusion} is not an equality in general.
For example, consider $\mathfrak{g}=B_2, \l=\rho, a=2$.  
In weight coordinates we have
\begin{align}
S_{\rho,2} &=  \cup_{\s\in W}(\s(\rho)-\rho+2\Pi_\rho)\cap \Lambda^+\\
&=  \{(2,2),(0,4),(3,0),(2,0),(0,2),(1,0),(0,0)\} 
\quad (\text{see figure below})
\end{align}
$$
\psdraw{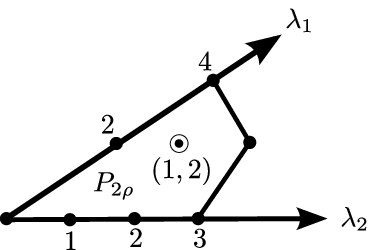}{2in}
$$
It is clear that $(1,2)\in P_{2\rho}$. We show that 
$(1,2)=\l_1+2\l_2\in \calL_{\rho,2}$ and hence this a missing. 
Indeed, by the definition of $\calL_{\rho,2}$, 
we only need to find $\s\in W$ and a root $\a$ such that
\begin{equation}
\lbl{eq.B2}
\l_1+2\l_2=2\a_1+3\a_2=2\rho+\s(\rho)-\rho+2\a
\end{equation}
In root coordinates we have 
$$
2\rho=3\a_1+4\a_2, \, \rho-\s(\rho)\in 
\{0,\a_1,\a_2,\a_1+3\a_2,2\a_1+\a_2,3\a_1+4\a_2\}
$$
So by choosing $\a=\a_2$ and $\s$ such that $\rho-\s(\rho)=\a_1+3\a_2$ we 
have equality \eqref{eq.B2}.
\end{remark}

Nevertheless, equality holds when $\mathfrak{g}=A_2,a=2,\l=\l_1$.
This is the content of the next section.

\subsection{A special case: no missing points}
\lbl{sub.nomissing}

\begin{proposition}
\lbl{prop.equality}
For $A_2$, we have: $S_{n\l_1,2}=\calL_{n\l_1,2}\cap P_{2n\l_1}$.
\end{proposition}

\begin{proof}
Let $\calL_{n\l_1,2}\cap P_{2n\l_1}\ni \mu=2n\l_1+\sigma(\rho)-\rho+2\a =
2\nu-(\rho-\sigma(\rho))$ where $\nu=n\l_1+\a$ and some $\sigma\in W$.
As $\sigma$ runs over the Weil group $W$, $\rho-\sigma(\rho)$ is expressed
in weight and root coordinates as follows
\be
\lbl{eq.weight.root}
\begin{array}{|c|cccccc|} \hline
\text{weight} &  (0,0) & (2,-1) & (-1,2) & (0,3) & (3,0) & (2,2)
\\ \hline 
\text{root} &  (0,0) & (0,1) & (1,0) & (1,2) & (2,1) & (2,2)
\\ \hline 
\end{array}
\ee
Since $\mu\in P_{2n\l_1}$, from inequality \eqref{ineq1} we have
$(2\nu-(\rho-\sigma(\rho)),\a_i)\geq 0$, i.e., $2(\nu,\a_i)\geq 
(\rho-\sigma(\rho),\a_i)$, $i=1,2$. Looking at the first row of the above 
table we see that this forces $(\nu,\a_i)\geq 0$, $i=1,2$. Therefore we 
have $\nu\in\Lambda^+$.

From inequality \eqref{ineq1} we have $(2n\l_1-\mu,\l_i)\geq 0$, $i=1,2$. This
implies that $(-2\a+\rho-\s(\rho),\l_i)\geq 0$ or equivalently
\begin{equation}
\lbl{ineq.root}
(\rho-\s(\rho),\l_i)\geq 2(\a,\l_i),
\end{equation}
$i=1,2$.
We consider the following cases.

{\bf Case 1:} 
If $\rho-\s(\rho)=0,\a_1$ or $\a_2$ then inequalities \eqref{ineq.root} 
imply that $(\a,\l_i)\leq 0$ for all $i$ so $\a \prec 0$. Since 
$\nu=n\l_1+\a\in\Lambda^+$, it follows from \cite[$\S$13.4]{Hu} 
that $\nu\in\Pi_{n\l_1}$ and hence $\mu\in S_{n\l_1,2}$.

{\bf Case 2:} 
If $\rho-\s(\rho)=\a_1+2\a_2$ then from \eqref{ineq.root} we have 
$(\a,\l_1)\leq 0$ and $(\a,\l_2)\leq 1$. 
If we also have $(\a,\l_2)\leq 0$ then  by a similar the argument to Case 1 
we conclude 
that  $\mu\in S_{n\l_1,2}$. If $(\a,\l_2)=1$ we can write 
$\a=-x\a_1+\a_2$, where $x\in\BN$.  It follows that 
$\mu=2n\l_1+2\a-(\rho-\s(\rho))=2n\l_1-2x\a_1+2\a_2-\a_1-2\a_2
=2(n\l_1-x\a_1)-\a_1$. Since $\nu=n\l_1+\a\in\Lambda^+$, from inequality 
\eqref{ineq1}
we have $(n\l_1-x\a_1+\a_2,\a_1)\geq 0$, i.e., $n-2x-1\geq 0$.
We have $\langle n\l_1,\a_1 \rangle=\frac{2(n\l_1,\a_1)}{(\a_1,\a_1)}
=n\geq 2x+1>x$, therefore $n\l_1-x\a_1\in \Pi_{n\l_1}$ 
(see \cite[$\S$ 13.4]{Hu}). Since we can choose $\s'$ such that 
$\rho-\s'(\rho)=\a_1$, we have $\mu=2(n\l_1-x\a_1)-(\rho-\s'(\rho))\in 
S_{n\l_1,2}$.

{\bf Case 3:} 
If $\rho-\s(\rho)=2\a_1+\a_2$ then by a similar argument to the above we can write $\a=\a_1-x\a_2,\, x\in\BN$. We show that 
$\a$ cannot have this form. Indeed, since $\nu=n\l_1+\a\in\Lambda^+$, 
we have $(n\l_1+\a_1-x\a_2,\a_2)\geq 0$, i.e., $-1-2x\geq 0$. This is in 
contradiction to the fact that $x\in\BN$.

{\bf Case 4:}
If $\rho-\s(\rho)=2\a_1+2\a_2=2\rho$ then $(\a,\l_1)\leq 1$ and 
$(\a,\l_2)\leq 1$.
If either $(\a,\l_1)\leq 0$ or $(\a,\l_2)\leq 0$ then the same
argument as in Cases 2 and 3 above apply. If $(\a,\l_1)=(\a,\l_2)=1$ then 
$\a=\a_1+\a_2=\rho$ and $\mu=2n\l_1+2\a-(\rho-\s(\rho))=2n\l_1
+2\rho-2\rho=
2n\l_1\in \Pi_{2n\l_1}\subseteq S_{n\l_1,2}$.
\end{proof}

\subsection{An estimate for the missing points}
\lbl{sub.missing}

The next proposition shows that the norm of the missing points in $R_{n\l,a}$
is bounded below by a quadratic function of $n$.

\begin{proposition}
\lbl{prop.complement}
For every $\l \in \La^+$ there exists a simple root $\beta$ such that
if $\mu\in R_{n\l,a}$ and $n \gg 0$ then 
$$
(\mu,\mu) \geq a^2n^2((\l,\l)-\frac{(\l,\b)^2}{(\beta,\beta)}-1) \,.
$$
\end{proposition}

\begin{proof}
Let $\mu=a\a+an\l+\s(\rho)-\rho=a(n\l+\a)+\s(\rho)-\rho$ for some 
$\a\in\Lambda_r$ and $\sigma\in\W$. Since $\mu\not\in S_{n\lambda,a}$,
we have that $n\l+\a\not\in \Pi_{n\l}$. The ray $n\lambda+\alpha$ 
meets one of the facets of the convex hull of $\Pi_{n\lambda}$ at some point, 
say $\lambda_n$. There exist $\sigma_1,\sigma_2\in\W$ such that 
 $\sigma_1(n\lambda),\sigma_2(n\lambda)$ are the vertices of this facet, and 
we have
\begin{align*}
(n\l+\a,n\l+\a) &\geq   (\l_n,\l_n)\\
&\geq  (\frac{\sigma_1(n\l)+\sigma_2(n\l)}{2},\frac{\sigma_1(n\l)
+\sigma_2(n\l)}{2})\\
&= \frac{n^2}{4}((\sigma_1(\l),\sigma_1(\l))+(\sigma_2(\l),\sigma_2(\l))
+2(\sigma_1(\l),\sigma_2(\l)))\\
&= \frac{n^2}{4}(2(\l,\l)+2(\sigma_1(\l),\sigma_2(\l)))\\
&= \frac{n^2}{2}((\l,\l)+(\sigma_1(\l),\sigma_2(\l)))
\end{align*}

\begin{figure}[htpb]
$$
\psdraw{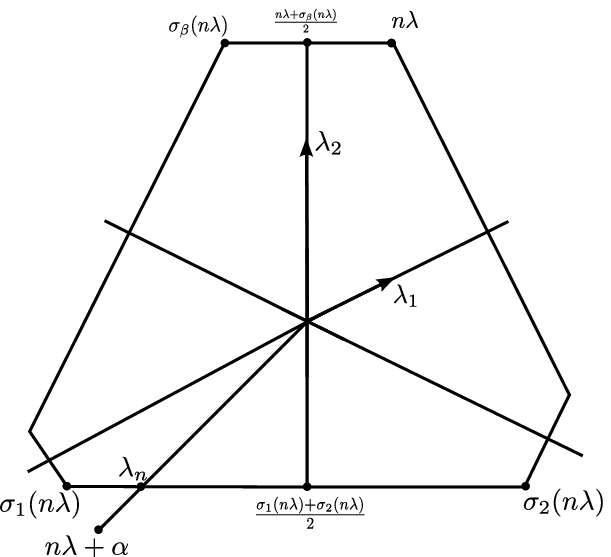}{3in}
$$
\end{figure}

Since $\sigma_1(\lambda),\sigma_2(\lambda)$ are in two nearby chambers, 
there exists a simple root $\beta$ such that 
$$
(\sigma_1(\l),\sigma_2(\l))=(\l,\sigma_\beta(\l))
$$
We have
\begin{align*}
(\l,\sigma_\beta(\l)) &= (\l,\l-2\frac{(\l,\b)}{(\b,\b)}\b)
= (\l,\l)-2\frac{(\l,\b)^2}{(\b,\b)}
\end{align*}
So
$$
(n\l+\a,n\l+\a) \geq n^2((\l,\l)-\frac{(\l,\b)^2}{(\beta,\beta)})
$$
Therefore 
\begin{align*}
(\mu,\mu) &= (a(n\l+\a)-(\rho-\sigma(\rho)),a(n\l+\a)-(\rho-\sigma(\rho))) \\
& = a^2(n\l+\a , n\l+\a) -2a(n\l+\a ,\rho-\sigma(\rho))
+(\rho-\sigma(\rho),\rho-\sigma(\rho))\\
&\geq  a^2n^2((\l,\l)-\frac{(\l,\b)^2}{(\beta,\beta)}-1)
\end{align*}
for large enough $n$.
\end{proof}

Recall that $\hat{S}_{\l,a}=S_{\l,a}-\mu_{\l,a}$. Let 
\begin{equation}
\lbl{eq.hats}
\hat{\calL}_{\l,a}=\calL_{\l,a}-\mu_{\l,a}, 
\qquad 
\hat{P}_{a\l}=P_{a\l}-\mu_{\l,a}, \qquad \hat{R}_{\l,a}=R_{\l,a}-\mu_{\l,a}. 
\end{equation}

\begin{remark}
\lbl{rm.ap}
From now we fix a natural number $n_0$ and we work with 
$n \equiv n_0 \bmod d a$ where $d$ is the order 
of the fundamental group $\Lambda/ \Lambda_r$. Theorem 
\ref{thm.min} implies that for such $n$, we have:
\begin{itemize}
\item $\mu_{n\l,a}=n\nu^1_{\l,a}+\nu^0_{\l,a}$ 
for some fixed weights $\nu^1_{\l,a}\, ,\nu^0_{\l,a}\in \Lambda^+$. 
\item $\hat{\calL}_{n\l,a}=\hat{\calL}_{n_0\l,a}$. Indeed, we have
\begin{align*}
\hat{\calL}_{n\l,a}&=\calL_{n\l,a}-\mu_{n\l,a}\\
&= n\l+\s(\rho)-\rho+a\Lambda_r-n\nu^1_{\l,a}-\nu^0_{\l,a}\\
&= n_0\l+\s(\rho)-\rho+a\Lambda_r-n_0\nu^1_{\l,a}-\nu^0_{\l,a}
+(n-n_0)(\l-\nu^1_{\l,a})\\
&= n_0\l+\s(\rho)-\rho+a\Lambda_r-\mu_{n_0\l,a}+k(a.d)(\l-\nu^1_{\l,a}),\, 
k\in\BN\\
&= n_0\l+\s(\rho)-\rho+a\Lambda_r-\mu_{n_0\l,a} \quad 
(\text{since}\,\,d(\l-\nu^1_{\l,a})\in\Lambda_r)\\
&= \calL_{n_0\l,a}-\mu_{n_0\l,a} = \hat{\calL}_{n_0\l,a}
\end{align*}
\end{itemize}
\end{remark}

\begin{corollary}
\lbl{main.cor}
\rm{(1)}
$\hat{S}_{n\lambda,a}\subset \hat{\calL}_{n_0\l,a} \cap\hat{P}_{an\lambda}$.
\newline
\rm{(2)} 
Let $\hat R_{n\l,a}=(\hat{\calL}_{n_0\l,a} \cap \hat{P}_{an\lambda})
\setminus \hat{S}_{n\lambda,a}$. If $\hat{\mu}\in \hat R_{n\l,a}$ then 
$$
(\hat{\mu},\hat{\mu})+2(\hat{\mu},\mu_{n\l,a}) \geq a^2n^2\left((\l,\l)
-\frac{(\l,\b)^2}{(\beta,\beta)}-1\right)-(\mu_{n\l,a},\mu_{n\l,a})
$$ 
for some simple root $\beta$.
\end{corollary}
\begin{proof}
Part $(1)$ follows from Lemma \ref{lem.Plambda}(b) and Remark \ref{rm.ap}:
\begin{align*}
\hat{S}_{n\lambda,a}& \subset \hat{\calL}_{n\l,a} \cap\hat{P}_{an\lambda}
=\hat{\calL}_{n_0\l,a} \cap\hat{P}_{an\lambda}
\end{align*}
For part $(2)$, recall that
$$(\mu,\mu)=(\hat{\mu}+\mu_{n\l,a},\hat{\mu}+\mu_{n\l,a})
$$ 
and therefore if $\hat{\mu}\in \hat R_{n\l,a}$ then
\begin{align*}
(\hat{\mu},\hat{\mu})+2(\hat{\mu},\mu_{n\l,a})
&=(\mu,\mu)-(\mu_{n\l,a},\mu_{n\l,a})\\
& \geq  a^2n^2\left((\l,\l)
-\frac{(\l,\b)^2}{(\beta,\beta)}-1\right)-(\mu_{n\l,a},\mu_{n\l,a})
\end{align*}
 by Proposition \ref{prop.complement}.
\end{proof}


\begin{proposition}
\lbl{main.cor1}
If $\fg$ has rank 2 and 
$\hat{\mu}\in \hat R_{n\l,a}$
then 
$$
(\hat{\mu},\hat{\mu})+2(\hat{\mu},\mu_{n\l,a}) \geq  n^2
$$
\end{proposition}
  
\begin{proof}
We can prove this by a direct computation for the rank $2$ simple Lie 
algebras $A_2, B_2$ and $G_2$ using Theorem \ref{thm.min} that gives
an explicit formula for $\mu_{\l,a}$.

For $A_2$ and $m_1\geq m_2$, from Theorem \ref{thm.min} we have
$$
(\mu_{n\l,a},\mu_{n\l,a})\leq (n(m_1-m_2)\l_2,n(m_1-m_2)\l_2)
=\frac{2}{3}n^2(m_1-m_2)^2
$$
By Corollary \ref{main.cor} we have 
\begin{align*}
(\hat{\mu},\hat{\mu})+2(\hat{\mu},\mu_{n\l,a}) &\geq  a^2n^2((\l,\l)
-\frac{(\l,\a_1)^2}{(\a_1,\a_1)}-1)-(\mu_{n\l,a},\mu_{n\l,a})\\
&\geq  a^2n^2(\frac{2}{3}(m_1^2+m_1m_2+m_2^2)-\frac{m_1^2}{2}-1)
-\frac{2}{3} n^2(m_1-m_2)^2\\
&= n^2(\frac{a^2-4}{6}m_1^2+\frac{2}{3}(a^2+2)m_1m_2
+\frac{2}{3}(a^2-1)m_2^2-1)\\
&\geq  n^2(4m_1m_2+2m_2^2-1)\\
& \geq  n^2
\end{align*}
except when $a=2$ and $m_2=0$. In the later case, Proposition
\ref{prop.equality} says that $R_{n\l,a}=\emptyset$ and the inequality holds
trivially. The argument is similar for 
the case $m_1\leq m_2$.

For $B_2$, $\frac{(\l,\beta)^2}{(\beta,\beta)}$ is either 
$\frac{m_1^2}{2}$ or $m_2^2$. We have
\begin{align*}
(\hat{\mu},\hat{\mu})+2(\hat{\mu},\mu_{n\l,a}) &\geq  
a^2n^2((\l,\l)-\frac{(\l,\a_1)^2}{(\a_1,\a_1)}-1)-(\mu_{n\l,a},\mu_{n\l,a})\\
&= a^2n^2(m_1^2+m_1m_2+\frac{m_2^2}{2}-\max\{\frac{m_1^2}{2},\frac{m_2^2}{4}\}-1)
-(\mu_{n\l,a},\mu_{n\l,a})\\
& \geq  n^2
\end{align*}
where in the last inequality we have used the fact that 
$(\mu_{\l,a},\mu_{\l,a})$ is bounded for $B_2$, 
see Theorem \ref{thm.min}. 

For $G_2$, $\frac{(\l,\beta)^2}{(\beta,\beta)}$ is either 
$\frac{m_1^2}{2}$ or $\frac{m_2^2}{6}$. Therefore we have
\begin{align*}
(\hat{\mu},\hat{\mu})+2(\hat{\mu},\mu_{n\l,a}) &\geq  a^2n^2((\l,\l)
-\frac{(\l,\a_1)^2}{(\a_1,\a_1)}-1)-(\mu_{n\l,a},\mu_{n\l,a})\\
&= a^2n^2(2m_1^2+6m_1m_2+6m_2^2-\max\{\frac{m_1^2}{2},
\frac{3m_2^2}{2}\}-1)-(\mu_{n\l,a},\mu_{n\l,a})\\
& \geq  n^2
\end{align*}
since $\mu_{\l,a}=0$  for $G_2$, see Theorem \ref{thm.min}.
\end{proof}


\section{Proof of Theorem \ref{thm.1}}
\lbl{sec.proof.thm1}

In this section we will prove Theorem \ref{thm.1} assuming Theorem
\ref{thm.min}. Corollary \ref{cor.minmax} implies that the shifted colored 
Jones polynomial defined by
\begin{equation}
\lbl{eq.hatJ}
\hat{J}^{\fg}_{T(a,b),\l}(q) =
q^{-\d^*_{T(a,b)}(\l)}J^{\mathfrak{g}}_{T(a,b),\l}(q)\in \mathbb{Z}[q] 
\end{equation}
satisfies
$$
\hat{J}^{\fg}_{T(a,b),\l}(q) = \frac{1}{\prod\limits_{\a\succ 0}(1-q^{(\l+\rho,\a)})}
\check{J}^{\fg}_{T(a,b),\l}(q)
$$
where
\be
\lbl{eq.Jcheck}
\check{J}^{\fg}_{T(a,b),\l}(q)=
\sum_{\hat{\mu} \in \hat{S}_{\l,a}} m^{\hat{\mu}+\mu_{\l,\a}}_{\l,a} 
q^{\frac{b}{2a}(\hat{\mu},\hat{\mu})+(-1+\frac{b}{a})(\hat{\mu},\rho)
+\frac{b}{a}(\hat{\mu},\mu_{\l,\a})}
\prod\limits_{\a\succ 0}(1-q^{(\hat{\mu}+\mu_{\l,a}+\rho,\a)})
\ee
with $\hat{S}_{\l,a}=S_{\l,a}-\mu_{\l,a}$ and $\hat{\mu}=\mu-\mu_{\l,a}$.

Fix a natural number $n$, observe that $(f_n(q))$ is c-stable if and only if 
$(f_{Mn+n_0}(q))$ is c-stable for all $n_0=0,1,...,M$. In what follows, we 
will use $M=ad$ and fix $n\equiv n_0$ mod $ad$.
\begin{proposition}
\lbl{main.cor2}
$(\hat{J}^{\fg}_{T(a,b),n\l}(q))$ is $c$-stable if and only if
\be
\lbl{eq.cor2}
\frac{1}{\prod\limits_{\a\succ 0}(1-q^{(n\l+\rho,\a)})}
\sum_{\hat{\mu} \in \hat{\calL}_{n_0\l,a} \cap\hat{P}_{an\lambda}}
m^{\hat{\mu}+\mu_{\l,\a}}_{\l,a} 
q^{\frac{b}{2a}(\hat{\mu},\hat{\mu})+(-1+\frac{b}{a})(\hat{\mu},\rho)
+\frac{b}{a}(\hat{\mu},\mu_{\l,\a})}
\prod\limits_{\a\succ 0}(1-q^{(\hat{\mu}+\mu_{\l,a}+\rho,\a)})
\ee
is $c$-stable. In that case, they have the same tails.
\end{proposition}

\begin{proof}
 Fix $a,b,\lambda$ and let $g_n(q)$ denote the difference between 
$\hat{J}_{T(a,b),n\lambda}(q)$ and Equation \eqref{eq.cor2}. Then
\begin{equation}
\lbl{eq.diff}
g_n(q) =
\frac{1}{\prod\limits_{\a\succ 0}(1-q^{(n\l+\rho,\a)})}
\sum_{\hat{\mu} \in \hat{R}_{n\l,a}}
m^{\hat{\mu}+\mu_{n\l,\a}}_{\l,a} 
q^{\frac{b}{2a}(\hat{\mu},\hat{\mu})+(-1+\frac{b}{a})(\hat{\mu},\rho)
+\frac{b}{a}(\hat{\mu},\mu_{n\l,\a})}
\prod\limits_{\a\succ 0}(1-q^{(\hat{\mu}+\mu_{n\l,a}+\rho,\a)})
\end{equation}
 Proposition \ref{main.cor1} implies that the minimum degree of the 
summands of Equation \eqref{eq.diff} is greater or equal to 
$\frac{b}{2a}n^2$ for $n\gg 0$. The proof then follows from Remark 
\ref{rm.cstbl}.

Proposition \ref{main.cor1} implies that we can replace the summation set 
$\hat{S}_{n\l,a}$ by $\hat{\calL}_{n\l,a} \cap\hat{P}_{an\lambda}$
without affecting the stability of  $\hat{J}^{\fg}_{T(a,b),n\l}(q)$: if 
$\hat{\mu}\in (\hat{\calL}_{\l,a}\cap \hat{P}_{an\lambda})\setminus 
\hat{S}_{n\lambda,a}$ then the minimum degree of the summand of Equation
\eqref{eq.cor2}
is 
\begin{align*}
\frac{b}{2a}(\hat{\mu},\hat{\mu})
+(-1+\frac{b}{a})(\hat{\mu},\rho)+\frac{b}{a}(\hat{\mu},\mu_{n\l,\a})
 &= \frac{b}{2a}((\hat{\mu},\hat{\mu})+2(\hat{\mu},\mu_{\l,a}))
-(\hat{\mu},\rho) \\
& \geq  \frac{b}{2a}n^2-(\hat{\mu},\rho)=\frac{b}{2a}n^2+O(n)
\end{align*} 
where the last inequality follows from Proposition \ref{main.cor1}.  By Remark
\ref{rm.ap} we have $\hat{\calL}_{n\l,a} =\hat{\calL}_{\l,a}$ and the Proposition 
follows.
\end{proof}

Let $t_{\l,\hat{\mu},a}(n)=m_{n\l,a}^{\hat{\mu}+\mu_{n\l,a}}$. Theorem 
implies that $t_{\l,\hat{\mu},a}$ is a quasi-polynomial.
From Lemma \ref{stability.lemma2}, Proposition \ref{main.cor1}, 
Proposition \ref{main.cor2} and together with the special case 
given in Section \ref{sec.sep} we conclude that

\begin{theorem}
\lbl{thm.stable.limit}
Fix a rank $2$ simple Lie algebras $\mathfrak{g}$, a dominant weight $\l$, 
and a torus knot $T(a,b)$. The colored Jones polynomial 
$\hat{J}^{\fg}_{T(a,b),n\l}(q)$ is $c$-stable and its $(n,x,q)$-tail is given by
\begin{align}
\lbl{stable.limit}
F_{T(a,b),\l}(n,x,q)&=
\frac{1}{\prod\limits_{\a\succ 0}(1-x^{(\l,\a)}q^{(\rho,\a)})}
\sum_{\hat{\mu} \in \hat{\calL}_{\l,a}\cap \Lambda^+} 
t_{\l,\hat{\mu},a}(n)
q^{\frac{b}{2a}(\hat{\mu},\hat{\mu})+(-1+\frac{b}{a})(\hat{\mu},\rho)
+\frac{b}{a}(\hat{\mu},\nu_{\l,a}^0)}x^{\nu_{\l,a}^1}
\\ \notag
&  \prod\limits_{\a\succ 0} (1-q^{(\hat{\mu}+\nu_{\l,a}^0+\rho,\a)}
x^{\nu_{\l,a}^1}) \,,
\end{align}
where $\mu_{n\l,a}=n\nu_{\l,a}^1+\nu_{\l,a}^0$.
\end{theorem}


\section{Proof of Theorem \ref{thm.max}}
\lbl{sec.thm.max}

In this section we prove Theorem \ref{thm.max}.
Since $\l$ is fixed, it suffices to maximize
$$
g(\mu) = \frac{b}{4}(\mu,\mu) +\left(-1+\frac{b}{2}\right)(\mu,\rho)
$$
on the set $S_{\l,a}$.

\begin{lemma}
\lbl{lem.positivity}
Let $\mu\in\Lambda^+$ and $\a \succ 0$ be a positive root 
such that $\mu+\a\in\Lambda^+$.  Then we have 
$$
(\mu,\mu)<(\mu+\a,\mu+\a)
$$
\end{lemma}

\begin{proof}
We have:
\begin{align*}
(\mu+\a,\mu+\a)-(\mu,\mu)=2(\mu,\a)+(\a,\a)
\end{align*}
Now $(\mu,\a)>0$ since $\mu$ is dominant and $\a$ is a positive root
and $(\a,\a)>0$ since $(\cdot,\cdot)$ is positive definite.
\end{proof}
If $\nu \in \Pi_\l$ then
$\nu=\l-\a'$ where $\a'\succ 0$. Since $\rho-\sigma(\rho)\succ 0$, we have
$\mu=a\l-\a$ where $\mu \in S_{a,\l}$ and  $\a\succ 0$. It follows from the 
above lemma that $M_{\l,a}=a\l$ is the unique maximizer of $f(\mu)$. 

Next, we compute the plethysm multiplicity $m_{\l,a}$.
From Lemma \ref{lem.mula} we have
\begin{align*}
m_{\l,a}^{a\l}  &= \sum\limits_{\sigma\in W} (-1)^\sigma 
m_\l^{\frac{a\l+\rho-\sigma(\rho)}{a}}\\
 &= \sum\limits_{\sigma\in W} (-1)^\sigma m_\l^{\l+\frac{\rho-\sigma(\rho)}{a}}\\
&= 1
\end{align*}
since $\l+ \frac{\rho-\sigma(\rho)}{a}\succ \l$ if 
$\frac{\rho-\sigma(\rho)}{a}\in\Lambda_r$,
with equality only when $\sigma=1$. This concludes the proof of Theorem
\ref{thm.max}.
\qed


\section{Proof of Theorem \ref{thm.min} }
\lbl{sec.thm.min}

This section is devoted to the proof of Theorem \ref{thm.min}, done by
a case-by-case analysis for a fixed simple Lie algebra $\fg$ of rank $2$.
Let $\l=m_1\l_1+m_2\l_2$ and $\mu=u_1\l_1+u_2\l_2$ be dominant weights. 
Since $\l$ is fixed, it suffices to minimize 
$$
g^*(\mu) = \frac{b}{4}(\mu,\mu)
+\left(-1+\frac{b}{2}\right)(\mu,\rho)
$$
on the set $S_{\l,a}$. We use the following lemma and its consequence, 
Corollary \ref{lem.unique},  in the proof
of Theorem \ref{thm.min}.

\begin{lemma}
\lbl{lem.pos}
$g^*(\mu) \geq 0$ with equality if and only if $\mu=0$.
\end{lemma}

\begin{proof}
$g^*(\mu)$ is non-negative since $(\cdot,\cdot)$ is a positive-definite 
form and $(\mu,\rho)\geq 0$ since $\mu$ is a dominant weight and $\rho$ is a 
linear combination of simple roots with positive coefficients. 
If $g^*(\mu)=0$ then $(\mu,\mu)=0$ which implies that $\mu=0$.
\end{proof}

\begin{corollary}
\lbl{lem.unique}
If $m_{\l,a}^0\neq 0$ then $\mu_{\l,a}=0$ is the unique minimizer of 
$g^*(\mu)$.
\end{corollary}
We give the proof of Theorem \ref{thm.min} in Section \ref{sub.A2} below.

\subsection{Theorem \ref{thm.min} for $A_2$}
\lbl{sub.A2}

\subsubsection{Plethysm multiplicities for $A_2$}
\lbl{sub.sub.A2}
There are two simple roots $\{\a_1,\a_2\}$ of $A_2$ and three positive
roots $\{\a_1,\a_1+\a_2,\a_2\}$ shown in Figure \ref{f.A2}.
The Kostant function $p(u,v)=p(u \a_1+v\a_2)$ is given by
$$
p(u,v)=1+\min(u,v) \,.
$$

\begin{figure}[htpb]
$$
\psdraw{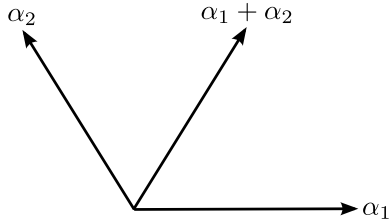}{2in}
$$
\caption{The two chambers of the Kostant partition function of $A_2$. 
Kostant chambers from left to right: $u\leq v$, $u\geq v$.}
\lbl{f.A2}
\end{figure}
Let $\l=m_1\l_1+m_2 \l_2$ denote a dominant weight and $m_1\geq m_2$.
Assuming $\mu=u_1\lambda_1+u_2\lambda_2\in\Pi_\l$, by Kostant's formula we have
\begin{align*}
m^\mu_\lambda &= \sum\limits_{\sigma\in W}(-1)^\sigma 
p(\sigma(\lambda+\rho)-\mu-\rho) \\ \notag
&=  p(\frac{2m_1+m_2}{3}-\frac{2u_1+u_2}{3},\frac{m_1+2m_2}{3}
-\frac{u_1+2u_2}{3}) \\ \notag
&  -
p(\frac{2m_1+m_2}{3}-\frac{2u_1+u_2}{3},\frac{m_1-m_2}{3}
-\frac{u_1+2u_2}{3}-1) \\
&=
\begin{cases}
1 + \frac{2m_1+m_2}{3}-\frac{2u_1+u_2}{3} & \text{if}~m_1-m_2<u_1-u_2\\
1 + \frac{m_1+2m_2}{3}-\frac{u_1+2u_2}{3} & 
\text{if}~u_1-u_2\leq m_1-m_2\leq u_1+2u_2+3\\
1+m_2 & \text{if}~ m_1-m_2 > u_1+2u_2+3
\end{cases}
\end{align*}
Lemma \ref{lem.mula} gives 
\begin{align*}
m^\mu_{\lambda,2} &=
\sum\limits_{\sigma\in S_3}(-1)^\sigma m^{\frac{\mu+\rho-\sigma(\rho)}{2}}_\lambda \\
&= m_\lambda^{\frac{1}{2}(u_1,u_2)}-m_\lambda^{\frac{1}{2}(u_1+2,u_2-1)}-
m_\lambda^{\frac{1}{2}(u_1-1,u_2+2)}+m_\lambda^{\frac{1}{2}(u_1,u_2+3)}+
m_\lambda^{\frac{1}{2}(u_1+3,u_2)}-m_\lambda^{\frac{1}{2}(u_1+2,u_2+2)}
\end{align*}

Let us consider $\mu\in S_{\l,2}$. There are four cases.

\noindent
{\bf Case 1:} $u_1, u_2$ are even.  
\begin{align}
\lbl{case1}
m^{\mu}_{\l,2}&= m^{(\frac{u_1}{2},\frac{u_2}{2})}_{\l}
-m^{(\frac{u_1+2}{2},\frac{u_2+2}{2})}_{\l}
=
\begin{cases}
1 & \text{if}~  u_1+2u_2\geq 2(m_1-m_2)\\
0 & \text{if}~  u_1+2u_2 < 2(m_1-m_2)
\end{cases}
\end{align}
{\bf Case 2:} $u_1$ even and $u_2$ odd. 
\begin{align}
\lbl{case2}
 m^{\mu}_{\l,2} &= m^{(\frac{u_1}{2},\frac{u_2+3}{2})}_{\l}
-m^{(\frac{u_1+2}{2},\frac{u_2-1}{2})}_{\l}\\
\notag
&=
\begin{cases}
-1 & \text{if}~ u_1-u_2 \leq 2(m_1-m_2)\leq u_1+2u_2\\
0 & \text{if}~   2(m_1-m_2) < u_1-u_2~\text{or}~ 2(m_1-m_2)>u_1+2u_2 
\end{cases}
\end{align}
{\bf Case 3:} $u_1$ odd and $u_2$ even.
\begin{align*}
 m^{\mu}_{\l,2}= m^{(\frac{u_1+3}{2},\frac{u_2}{2})}_{\l}
-m^{(\frac{u_1-1}{2},\frac{u_2+2}{2})}_{\l} =
\begin{cases}
-1 & \text{if}~  2(m_1-m_2)<u_1-u_2\\
0 & \text{if}~   2(m_1-m_2)\geq u_1-u_2 
\end{cases}
\end{align*}
{\bf Case 4:} $u_1$ and $u_2$ are odd.
\begin{equation*}
\lbl{case4}
 m^{\mu}_{\l,2}=0 
\end{equation*}

\begin{corollary}
\lbl{cor.nonzero}
For $A_2$, if $m_{\l,2}^\mu\neq 0$ then $u_1+2u_2\geq 2(m_1-m_2)$.
\end{corollary}
If $m_1\leq m_2$ we have a similar corollary:
\begin{corollary}
\lbl{cor.nonzero1}
For $A_2$, if $m_{\l,2}^\mu\neq 0$ then $2u_1+u_2\geq 2(m_2-m_1)$.
\end{corollary}
\subsubsection{The minimizer for $A_2$}~\\
{\bf Case 1:} $a=2$. By Corollary \ref{cor.nonzero} it suffices to minimize
$g^*(\mu)$ over subset 
$\{\mu\in S_{\l,2}:u_1,u_2\in\BN,u_1+2u_2\geq 2(m_1-m_2)\}$ of $S_{\l,2}$. 
We have
\begin{align*}
g^*(\mu) &= \frac{b}{4}(\mu,\mu)
+\left(-1+\frac{b}{2}\right)(\mu,\rho)\\
&=  \frac{b}{6}(u_1^2+u_1u_2+u_2^2)+(-1+\frac{b}{2})(u_1+u_2)\\
&= \frac{b}{6}((u_2+\frac{u_1}{2})^2+\frac{3u_1^2}{4})
+\frac{b-2}{4}(u_1+u_1+2u_2)\\
&\geq \frac{b}{8}u_1^2+\frac{b-2}{4}u_1+\frac{b}{6}(m_1-m_2)^2
+\frac{b-2}{2}(m_1-m_2)\\
&\geq  \frac{b}{6}(m_1-m_2)^2+\frac{b-2}{2}(m_1-m_2)
\end{align*}
with equality if and only if $u_1=0,u_2=m_1-m_2$. 

Next we show that $\mu_{\l,2}=(m_1-m_2)\l_2\in S_{\l,2}$. Indeed, 

\begin{enumerate}
\item If $m_1-m_2\equiv 0$ (mod 2) then 
$\mu_{\l,2}=2\nu-(\rho-\sigma(\rho))\in S_{\l,2}$ where 
$\nu=\frac{m_1-m_2}{2}\l_2\in \Pi_\l$ and $\sigma=1$.
\item If $m_1-m_2\equiv 1$ (mod 2) then 
$\mu_{\l,2}=2\nu-(\rho-\sigma(\rho)) \in S_{\l,2}$ where 
$\nu=\frac{m_1-m_2+3}{2}\l_2\in \Pi_\l$ and $\rho$ such that
$\rho-\sigma(\rho)=3\l_2$.
\end{enumerate}
Note that from the formula for $m_{\l,2}^\mu$ in 
Equations \eqref{case1} and \eqref{case2} we have 
$m_{\l,2}^{(m_1-m_2)\l_2}=1$ which proves part (a). Part (b) is obvious.
The case $m_1\leq m_2$ is similar.

\noindent
{\bf Case 2:} $a=3$.  
\lbl{case2again}
From Equation \eqref{eq.Sla}, we have
$$
m^0_{\lambda,3}=m^0_\lambda+m^{\lambda_1}_\lambda+m^{\lambda_2}_\lambda
$$
Since the fundamental group for $A_2$ consists of only three elements
 $0,~\lambda_1,~\lambda_2$, at least one of the terms on the right hand 
side is greater than zero. Therefore $m^0_{\lambda,3}>0$ and it follows from 
Lemma \ref{lem.pos} that $\mu_{\l,3}=0$ for all $\lambda$. 
Therefore part (b) follows. 
Part (a) follows from Corollary \ref{lem.unique} and the fact that 
$m^0_{\lambda,3}>0$.\\
\noindent
{\bf Case 3:} $a \geq 4$. 
\lbl{typical.case} 

{\bf Claim}. At most one term on the right hand side of Equation 
\eqref{eq.mula} is nonzero.

\begin{proof}
 Indeed, if there are $\sigma_1,\sigma_2$ in the Weyl group for 
$A_2$ such that
$m_{\l}^{\frac{\mu+\rho-\sigma_1(\rho)}{a}}\neq 0$  and 
$m_{\l}^{\frac{\mu+\rho-\sigma_2(\rho)}{a}}\neq 0$ 
 then $\frac{\mu+\rho-\sigma_1(\rho)}{a}
-\frac{\mu+\rho-\sigma_2(\rho)}{a}\in\Lambda_r$. 
 Equivalently, $(\rho-\sigma_1(\rho))-(\rho-\sigma_2(\rho))\in a\Lambda_r$. 
This is a contradiction since $a\geq 4$ and by \cite{Ko1},
$$
\rho-\sigma(\rho)=\sum\limits_{\a\in\Delta^{+}:\sigma^{-1}(\a)\in \Delta^{-}}\a
$$
which do not belong to $a\Lambda_r$ if $a\geq 4$. Here $\Delta^+$ is the set 
of positive roots and $\Delta^-=-\Delta^+$. 
\end{proof}

\noindent
{\bf Case 3.1:}
$\l\in\Lambda_r$, i.e., $m_1-m_2\equiv 0 \bmod 3$. By the above claim we 
have $m^0_{\l,a}=m_\l^{\frac{\rho-\sigma(\rho)}{a}}$ for some $\sigma$. It's easy 
to see that the only $\sigma$ for which $\frac{\rho-\sigma(\rho)}{a}$ is a 
weight is when $\sigma=1$ and therefore $m_{\l,a}^0=m_\l^0>0$. It follows from 
Lemma \ref{lem.pos} that $\mu_{\l,a}=0$. Therefore part (b) follows for this 
case. Part (a) follows from Corollary \ref{lem.unique} and the fact that 
$m^0_{\lambda,3}>0$.\\

\noindent
{\bf Case 3.2:} 
If $\l\not\in \Lambda_r$, or equivalently $m_1-m_2\not\equiv 0 
\bmod 3$ then $m_{\l,a}^0=m_\l^0=0$
so $\mu_{\l,a}\neq 0$. 
By the above claim, we have
$$
m_{\l,a}^\mu=(-1)^{\s} m_{\l}^{\frac{\mu+\rho-\sigma(\rho)}{a}}
$$
for some$\sigma$. Furthermore,   $m_{\l}^{\frac{\mu+\rho-\sigma(\rho)}{a}}\neq 0$ 
if and only if $\frac{\mu+\rho-\sigma(\rho)}{a}=\nu\in\Pi_\l$ or equivalently, 
$\mu=a\nu-(\rho-\sigma(\rho))$.
Let $\rho-\sigma(\rho)=s\lambda_1+t\lambda_2$, where 

\be
\lbl{eq.orbit}
\begin{array}{|c|cccccc|} \hline
(s,t) & 
(0,0) & (-1,2) & (1,-2) & (0,3) & (3,0) & (2,2)
\\ \hline 
(-1)^{\s} & 1 & -1 & -1 & 1 & 1 & -1
\\ \hline 
\end{array}
\ee
So if $\nu=v_1\l_1+v_2\l_2$ then $\mu=(av_1-s)\l_1+(av_2-t)\l_2$. Since 
$\mu$ is a positive weight, we have
we have
\begin{align*}
av_1-s &\geq  0\\
av_2-t & \geq 0
\end{align*}
Since $a\geq 4$ and $|s|, |t| \leq 3$, these inequalities imply that $v_1,v_2\geq 0$, i.e.,
$\nu$ is also a positive weight. There are two possibilities for $\l$.

\noindent
{\bf Case 3.2.1:}
$\l_1\in\Pi_\l$, i.e., $m_1\equiv m_2+1 \bmod 3$. Then we can 
choose $\nu_0=\l_1$ and $\sigma_0$ to be the unique element
 in $\W$ such that $\rho-\sigma_0(\rho)=3\l_1$. We will prove that
  $\mu_{\l,a}=a\nu_0-(\rho-\sigma_0(\rho))=(a-3)\l_1$ is the minimizer. 
Indeed, let $\mu=a\nu-(\rho-\sigma(\rho))\in S_{\l,a}$ where 
$\nu\in\Pi_\l$ as above.

\noindent
{\bf Case 3.2.1.1:}
If $\nu=\l_1$ then for $\mu$ to be a dominant weight we should have, 
according to Table \eqref{eq.orbit},
\begin{align*}
\rho-\sigma(\rho)=
\begin{cases}
0 &\text{which gives}\,\,\, \mu=a\l_1\\
\l_1-2\l_2  &\text{which gives}\,\,\, \mu=(a-1)\l_1+2\l_2\\
3\l_1 &\text{which gives}\,\,\, \mu=(a-3)\l_1=\mu_{\l,a}
\end{cases}
\end{align*}
It is easy to check that $g^*(\mu)> g^*(\mu_{\l,a})$ for the first two values 
of $\mu$.

\noindent
{\bf Case 3.2.1.2:}
 If $\nu\neq \l_1$, let $\nu=v_1\l_1+v_2\l_2$ then we have $v_1,v_2\geq 0$ and 
$v_1+v_2\geq 3$, since the only cases where $v_1+v_2 < 3$ are  
$\nu=\l_2$ and $\l_1+\l_2$ but these weights donot belong in $\Pi_\l$. Let 
$\nu=a\nu-(\rho-\sigma(\rho))=(av_1-s)\l_1+(av_2-t)\l_2$ as before.
 We have
\begin{align*}
g^*(\mu)&=\frac{b}{2a}(\mu,\mu)
+\left(-1+\frac{b}{a}\right)(\mu,\rho)\\
&= \frac{b}{3a}(u_1^2+u_1u_2+u_2^2)+(-1+\frac{b}{a})(u_1+u_2)\\
&= \frac{b}{3a}(a^2(v_1^2+v_1v_2+v_2^2)-2a(v_1+v_2)(s+t)+s^2+st+t^2)\\
& +  (-1+\frac{b}{a})(a(v_1+v_2)-s-t)
\end{align*}

It is easy to check that for all 
$(s,t)\in \{(0,0),(-1,2),(1,-2),(0,3),(3,0),(2,2)\}$ and 
$(v_1,v_2): v_1, v_2 \geq 0, v_1+v_2\geq 3$, we have

\begin{align*}
a^2(v_1^2+v_1v_2+v_2^2)-2a(v_1+v_2)(s+t)+s^2+st+t^2 & >  (a-3)^2\\
a(v_1+v_2)-s-t & > a-3
\end{align*}
and therefore $g^*(\mu) > \frac{b}{3a}(a-3)^2
+(-1+\frac{b}{a})(a-3)=g^*(\mu_{\l,a})$ for all $\mu\neq\l_1$.

The above argument showed that $\mu_{\l,a}=(a-3)\l_1$ is the unique minimizer, 
and note that $m_{\l,a}^{(a-3)\l_1}=m_\l^{\l_1}\neq 0$ since 
$\l_1\in\Pi_\l$. This proves parts (a) and (b) for Case 3.2.1. \\
{\bf Case 3.2.2:}
$\l_2\in \Pi_\l$ or equivalently, $m_1\equiv m_2+2 \bmod 3$. The 
proof for this is identical to the one above.

This completes the proof of Theorem \ref{thm.min} for $A_2$.
\qed

\subsection{Theorem \ref{thm.min} for $B_2$}
\lbl{sub.B2}

There are two simple roots $\{\a_1,\a_2\}$ and four positive roots
$\{a_1,\a_2,\a_1+\a_2,\a_1+2\a_2\}$ of $B_2$ shown in Figure \ref{f.B2}.
The Kostant partition function $p(u,v)=p(u \a_1+v \a_2)$ is given by \cite{Ta}
\begin{align}
\lbl{eq.kostantB2}
p(u,v)&=
\begin{cases}
b(v) &\text{if $u\geq v$}\\
b(v)-\frac{(v-u)(v-u+1)}{2}   &\text{if $u\leq v\leq 2u$}\\
\frac{(u+1)(v+2)}{2}  &\text{if $2u\leq v$}
\end{cases},
\qquad
b(n)=\frac{n^2}{4}+n+
\begin{cases}
1 &\text{if $2 | n$} \\
\frac{3}{4} &\text{if $2 \not{|} n$}
\end{cases} \,.
\end{align}
There are three Kostant chambers shown in Figure \ref{f.B2}.
\begin{figure}[htpb]
$$
\psdraw{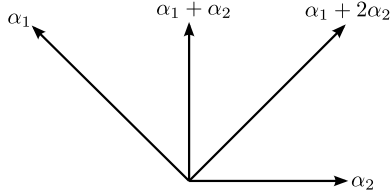}{2in}
$$
\caption{The three chambers of the Kostant partition function of $B_2$.
Kostant chambers from left to right: $u\geq v$, $u\leq v\leq 2u$, $u\geq 2v$.}
\lbl{f.B2}
\end{figure}
Let $\l=m_1\l_1+m_2\l_2$ denote a dominant weight. 
In weight coordinates we have
$$
\rho-\s(\rho)=s \l_1+t \l_2
$$ 
where
\be
\lbl{eq.orbitB2}
\begin{array}{|c|cccccccc|} \hline
(s,t) & 
(0, 0) & (2,-2) & (-1, 2) & (-1, 4) & (3,-2) & (3,0) & (0,4) & (2,2)
\\ \hline 
(-1)^{\s} & 1 & -1 & -1 & 1 & 1 & -1 & -1 & 1
\\ \hline 
\end{array}
\ee
Lemma \ref{lem.mula} implies that
\begin{align}
\lbl{eq.B2mu0}
m_{\l,a}^0 = m_\l^0 + \begin{cases}
-m_\l^{\l_2}-m_\l^{2\l_2}+m_\l^{\l_1+\l_2} &\text{if} \quad a=2\\
-m_\l^{\l_1} &\text{if} \quad a=3\\
-m_\l^{\l_2} &\text{if} \quad a=4\\
0 &\text{if} \quad a\geq 5
\end{cases}
\end{align}

\noindent
{\bf Case 1:} $a=2$. Equation \eqref{eq.B2mu0} implies that
\begin{equation}
\lbl{eq.B2case1}
m^0_{\lambda,2}=m^0_\lambda-m^{2\lambda_2}_\lambda 
-m^{\lambda_2}_\lambda+m^{\lambda_1+\lambda_2}_\lambda \,.
\end{equation}

\noindent
{\bf Case 1.1:}
$\l\in\Lambda_r$, i.e., $m_2\equiv 0 \bmod 2$.  In this case, we have 
$\lambda_1+\lambda_2$, $\lambda_2\not\in\Lambda_r$, and therefore
$m^{\lambda_2}_\lambda=m^{\lambda_1+\lambda_2}_\lambda=0$. Equation
\eqref{eq.B2case1} becomes
$$
m^0_{\lambda,2}=m^0_\lambda-m^{2\lambda_2}_\lambda =1 \,.
$$
where the later equality comes from formula \eqref{eq.kostantB2} and the 
Kostant multiplicity formula \eqref{eq.komult}.
It follows from Lemma \ref{lem.pos} that $\mu_{\l,2}=0$ which proves part (b). 
Part (a) follows from Corollary \ref{lem.unique} and the fact that 
$m^0_{\lambda,2}=1\neq 0$.\\

\noindent
{\bf Case 1.2:}
$\l\not\in\Lambda_r$, i.e., $m_2\equiv 1 \bmod 2$. Since 
$m^0_\lambda=m^{2\lambda_2}_\lambda=0$ we have
$$
m^0_{\lambda,2} = m_\lambda^{\lambda_1+\lambda_2}-m_\lambda^{\lambda_2} 
=-1 \,.
$$

If $m_1>0$ then choose $\nu=\l_1+\l_2\in\Pi_\l$ and $\sigma$ such that 
$\rho-\sigma(\rho)=2\l_1+2\l_2$ we obtain 
$\mu_{\l,2}=2\nu-(\rho-\sigma(\rho))=2(\l_1+\l_2)-(2\l_1+2\l_2)=0$. 
If otherwise $m_1=0$ then we choose $\nu=\l_2\in\Pi_\l$, $\s$ such that 
$\rho-\sigma(\rho)=-\l_1+2\l_2$ and get $\mu_{\l,2}=2\l_2-(-\l_1+2\l_2)=\l_1$. 
This  proves part (b). Part (a) follows from Corollary \ref{lem.unique} and 
the fact that $m^0_{\lambda,2}=-1\neq 0$.

\noindent
{\bf Case 2:} $a=3$. Consider two small cases. \\
{\bf Case 2.1:}
If $\lambda=m_1\l_1+m_2\l_2\in\Lambda_r$, i.e., $m_2\equiv 0 \bmod 2$ 
then we have
\begin{align*}
m^0_{\lambda,3}= m^0_\lambda-m_\lambda^{\lambda_1}
= \frac{1}{2}+\frac{(-1)^{m_1+m_2}+(-1)^{m_1+m_2+2}}{4}
\end{align*}

If $m_1\equiv 0 \bmod 2$ then $ m^0_{\lambda,3}=1$. 
It follows from Lemma  \ref{lem.pos}  that $\mu_{\l,3}=0$ and  this 
completes part (b). 
Part (a) follows from Corollary \ref{lem.unique} and the fact that 
$m^0_{\lambda,3}=1\neq 0$.

If  $m_1\equiv 1  \bmod 2$ then $ m^0_{\lambda,3} = 0$. By a similar 
argument to the one in
Case 3 for $A_2$ it can be shown  that $\mu_{\l,3}=2\l_2$ is the 
unique minimizer and parts (a) and (b) follow. \\
\noindent
{\bf Case 2.2:}
If $\l=m_1\l_1+m_2\l_2\not\in\Lambda_r$, i.e., $m_2\not\equiv 0 \bmod 2$
then by a similar argument to the one in Case 3 for $A_2$ we have 
$\mu_{\l,3}=\l_1+\l_2$ is the unique minimizer  and $m_{\l,3}^{\l_1+\l_2}\neq 0$ 
which completes the proof.

\noindent
{\bf Case 3:} $a=4$. From Equation \eqref{eq.Sla} we have 
$$
m^0_{\lambda,4}=m^0_{\lambda}-m^{\lambda_2}_{\lambda}
$$

If $\lambda=m_1\l_1+m_2\l_2\in\Lambda_r$, i.e., $m_2\equiv 0 \bmod 2$  then 
we have $m^0_{\lambda,4}=m^0_{\lambda}-m^{\lambda_2}_{\lambda}=m^0_{\lambda}>0$, since 
$0\in\Pi_\l$.
 
If $\l=m_1\l_1+m_2\l_2\not\in\Lambda_r$, i.e., $m_2\not\equiv 0 \bmod 2$ then  
$m^0_{\lambda,4}=m^0_{\lambda}-m^{\lambda_2}_{\lambda}=-m^{\lambda_2}_{\lambda}<0$, 
since $\l_2\in\Pi_\l$. 

It follows from Lemma \ref{lem.pos} that $\mu_{\l,4}=0$, which completes 
part (b). Part (a) follows from Corollary \ref{lem.unique}.

\noindent
{\bf Case 4:} $a\geq 5$. The only $\sigma$ for which 
$\mu=\frac{\rho-\sigma(\rho)}{a}$ is a weight is 
$\sigma=1$ and hence $\mu=0$. So from Equation \eqref{eq.Sla} we have 
$m_{\l,a}^0=m_\l^0$.

If $\lambda=m_1\l_1+m_2\l_2\in\Lambda_r$, i.e., $m_2\equiv 0 \bmod 2$  
then $m_{\l,a}^0=m_\l^0>0$. It follows from Lemma \ref{lem.pos} that 
$\mu_{\l,a}=0$, which completes part (b). Part (a) follows from Corollary 
\ref{lem.unique}.

If $\l=m_1\l_1+m_2\l_2\not\in\Lambda_r$, i.e., $m_2\not\equiv 0 \bmod 2$  
then by a similar argument to the one in Case 3 for $A_2$ we have 
that $\mu_{\l,a}=(a-4)\l_2$ is the unique minimizer and
$m_{\l,a}^{(a-4)\l_2}=m_{\l}^{\l_2}\neq 0$. This completes both parts (a) and (b).

This completes the proof of Theorem \ref{thm.min} for $B_2$.
\qed
  
\subsection{Theorem \ref{thm.min} for $G_2$}
\lbl{sub.G2}

There are two simple roots $\{\a_1,\a_2\}$ and six positive roots
$\{a_1,\a_2,\a_1+\a_2,2\a_1+\a_2,3\a_1+\a_2,3\a_1+2\a_2\}$ of $G_2$ shown 
in Figure \ref{f.G2}.
\begin{figure}[htpb]
$$
\psdraw{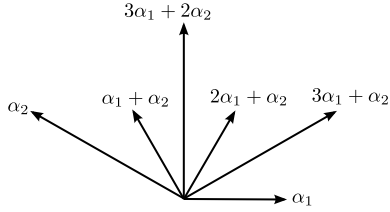}{2in}
$$
\caption{The five chambers of the Kostant partition function of $G_2$.
Kostant chambers from left to right: $u\leq v$, $v\leq u\leq \frac{3}{2}v$, 
$\frac{3}{2}v\leq u\leq 2v$, $2v\leq u\leq 3v$, $3v\leq u$.}
\lbl{f.G2}
\end{figure}
The Kostant partition function $p(u,v)=p(u \a_1+v \a_2)$ is given by 
\cite{Ta}
\begin{align}
p(u,v) =  \begin{cases}
g(u) &\text{if} \quad u\leq v\\
g(u)-h(u-v-1) &\text{if} \quad v\leq u\leq \frac{3}{2}v\\
h(v)-g(3v-u-1)+h(2v-u-2) &\text{if} \quad \frac{3}{2}v\leq u\leq 2v \\
h(v)-g(3v-u-1) &\text{if} \quad 2v\leq u \leq 3v \\
h(v) &\text{if} \quad 3v \leq u\\
\end{cases}
\end{align}  
where
\begin{align}
g(n) = \begin{cases}
\frac{1}{432}(n+6)(n^3+14n^2+54n+72) &\text{if} \quad n\equiv 0 \bmod 6\\
\frac{1}{432}(n+5)^2(n^2+10n+13) &\text{if} \quad n\equiv 1 \bmod 6\\
\frac{1}{432}(n+4)(n^3+16n^2+74n+68) &\text{if} \quad n\equiv 2 \bmod 6\\
\frac{1}{432}(n+3)^2(n+5)(n+9) &\text{if} \quad n\equiv 3 \bmod 6\\
\frac{1}{432}(n+2)(n+8)(n^2+10n+22) &\text{if} \quad n\equiv 4 \bmod 6\\
\frac{1}{432}(n+1)(n+5)(n+7)^2 &\text{if} \quad n\equiv 5 \bmod 6
\end{cases}
\end{align}
and
\begin{align}
h(n) = \begin{cases}
\frac{1}{48}(n+2)(n+4)(n^2+6n+6) &\text{if} \quad n\equiv 0 \bmod 2\\
\frac{1}{48}(n+1)(n+3)^2(n+5) &\text{if} \quad n\equiv 1 \bmod 2
\end{cases}
\end{align}

From Lemma \ref{lem.mula} we have
\begin{align}
m_{\l,a}^0 = m_\l^0 + \begin{cases}
-m_\l^{3\a_1+\a_2}-m_\l^{2\a_1+2\a_2}+m_\l^{5\a_1+3\a_2} &\text{if} \quad a=2\\
-m_\l^{3\a_1+2\a_2} &\text{if} \quad a=3\\
-m_\l^{\a_1+\a_2} &\text{if} \quad a=4\\
-m_\l^{2\a_1+\a_2} &\text{if} \quad a=5\\
0 &\text{if} \quad a\geq 6
\end{cases}
\end{align}
From now on, let us consider $\l=u\a_1+v\a_2\in\Lambda^+$, so 
$\frac{3}{2}v\leq u\leq 2v$. 

\noindent
{\bf Case 1:} $a=2$. We have
\begin{equation}
\lbl{g2.a2}
m_{\l,2}^0=m_\l-m_\l^{3\a_1+\a_2}-m_\l^{2\a_1+2\a_2}+m_\l^{5\a_1+3\a_2}
\end{equation}
Using the Kostant multiplicity formula we can calculate the weight 
multiplicities on the right hand side of Equation \eqref{g2.a2}, we
have, for example
\begin{align*}
m_{\lambda}^0 &= \sum\limits_{\sigma\in\W}(-1)^\sigma p(\sigma(\l+\rho)-\rho)\\
&= p(u,v)-p(-u+3v-1,v)-p(u,u-v-1)+p(3v-u-1,2v-u-2)\\
 &+  p(2u-3v-4,u-v-1)\\
&= \frac{u^4}{9}-\frac{29 u^3 v}{36}-\frac{7 u^3}{36}+\frac{17 u^2 v^2}{8}
+\frac{2 u^2 v}{3}
-\frac{19 u^2}{24}-\frac{29 u v^3}{12}-\frac{u v^2}{2}+3uv\\
&+ v^4-\frac{v^3}{12}-\frac{21 v^2}{8}+c_{1,0}(u)u+c_{0,1}(v)v+c_{0,0}(u,v)
\end{align*}
where 
\begin{align*}
c_{1,0}(u)=
\begin{cases}
\frac{1}{4} & \text{if $u\equiv 0$ mod 3}\\
\frac{17}{36} & \text{if $u\equiv 1$ mod 3}\\
\frac{25}{36} & \text{if $u\equiv 2$ mod 3}\\
\end{cases}
, \hspace*{0.3cm} c_{0,1}(u)=
\begin{cases}
\frac{1}{12} & \text{if $u\equiv 0$ mod 3}\\
-\frac{13}{36} & \text{if $u\equiv 1$ mod 3}\\
-\frac{29}{36} & \text{if $u\equiv 2$ mod 3}\\
\end{cases}
\end{align*}
\begin{align*}
c_{0,0}(u,v)=
\begin{cases}
1 & \text{if $u\equiv 0$ mod 6, $v\equiv 0$ mod 2}\\
\frac{29}{72} & \text{if $u\equiv 1$ mod 6, $v\equiv 0$ mod 2}\\
\frac{5}{9}& \text{if $u\equiv 2$ mod 6, $v\equiv 0$ mod 2}\\
\frac{5}{8} & \text{if $u\equiv 3$ mod 6, $v\equiv 0$ mod 2}\\
\frac{7}{9} & \text{if $u\equiv 4$ mod 6, $v\equiv 0$ mod 2}\\
\frac{13}{72} & \text{if $u\equiv 5$ mod 6, $v\equiv 0$ mod 2}
\end{cases}
, \hspace*{0.3cm}
c_{0,0}(u,v)=
\begin{cases}
\frac{5}{8} & \text{if $u\equiv 0$ mod 6, $v\equiv 1$ mod 2}\\
\frac{5}{18} & \text{if $u\equiv 1$ mod 6, $v\equiv 1$ mod 2}\\
\frac{13}{72}& \text{if $u\equiv 2$ mod 6, $v\equiv 1$ mod 2}\\
\frac{5}{8} & \text{if $u\equiv 3$ mod 6, $v\equiv 1$ mod 2}\\
\frac{29}{72} & \text{if $u\equiv 4$ mod 6, $v\equiv 1$ mod 2}\\
\frac{13}{72} & \text{if $u\equiv 5$ mod 6, $v\equiv 1$ mod 2}
\end{cases}
\end{align*}

$m_\l^{3\a_1+\a_2}, m_\l^{2\a_1+2\a_2}, m_\l^{5\a_1+3\a_2}$ can be computed similarly 
to show that $m_{\l,2}^0=1$. This confirms part (b). Part (a) follows Corollary 
\ref{lem.unique}.

\noindent
{\bf Case 2:} $a=3$. We have
\begin{align*}
m_{\lambda,3}^{0} =& m^0_\lambda-m^{[3,2]}_\lambda\\
=& -u^2+\frac{7 u v}{2}+\frac{u}{2}-3 v^2-\frac{v}{2}+c_{0,0}(u,v)
\end{align*}
where 
\begin{align*}
c_{0,0}(u,v)=
\begin{cases}
1 & \text{if $u\equiv 0,v\equiv 0$ mod 2}\\
\frac{1}{2} & \text{if $u\equiv 1,v\equiv 0$ mod 2}\\
\frac{1}{2} & \text{if $v\equiv 1$ mod 2}
\end{cases}
\end{align*}
Note that since $\frac{3v}{2}\leq u\leq 2v$, 
$-u^2+\frac{7 u v}{2}+\frac{u}{2}-3 v^2-\frac{v}{2}
=(-u^2+\frac{7 u v}{2}-3 v^2)+\frac{u-v}{2}\geq 0$ and therefore 
$m_{3,\lambda}^{0}>0$ for all $\lambda$. Part (a) follows from Lemma 
\ref{lem.pos} and part (b) follows from Corollary \ref{lem.unique}.

\noindent
{\bf Case 3:} $a=4,5$. The arguments are similar to that of Case 2.

\noindent
{\bf Case 4:} $a\geq 6$, can be done without computations. Indeed, we have 
$m_{\l,a}^0=m_\l^0 > 0$ since $\l\in\Lambda_r$; see \cite[$\S$13.4,Lem.B]{Hu}. 
Parts (a) and (b) follow from Lemma \ref{lem.pos} and 
Corollary \ref{lem.unique}.

This completes the proof of Theorem \ref{thm.min} for $G_2$.
\qed


\section{Some tails of the $T(2,3)$ and $T(4,5)$ torus knots}
\lbl{sec.examples}

\subsection{The tail for $A_2$ and the trefoil}
\lbl{sec.sep}

In this section we compute the tail of the $c$-stable sequence 
$J^{A_2}_{T(2,b),n\lambda_1}(q)$ for $b>2$ odd.
From Proposition \ref{thm.min} we have $\mu_{n\l_1,2}=n\l_2$ so
Equation \eqref{eq.hatJ} gives

\begin{equation*}
\hat{J}^{A_2}_{T(2,b),n\lambda_1}(q) =
\frac{1}{(1-q)(1-q^{n+1})(1-q^{n+2})} \check{J}^{A_2}_{T(2,b),n\lambda_1}(q)
\end{equation*}
where
\begin{align*}
\check{J}^{A_2}_{T(2,b),n\lambda_1}(q) = &
\sum\limits_{u_1\l_1+u_2\l_2\in S_{n\l_1,2}} c(u_1,u_2)
q^{\frac{b}{6}(u_1^2+u_1u_2+u_2^2-n^2)+(\frac{b}{2}-1)(u_1+u_2-n)} \\ & 
\cdot (1-q^{u_1+1})(1-q^{u_2+1})(1-q^{u_1+u_2+2})
\end{align*}
and from Cases 1-4 of Section \ref{sub.sub.A2},
\begin{align*}
 c(u_1,u_2)= m_{n\l_1,2}^{u_1\l_1+u_2\l_2}=
\begin{cases}
1 & \text{if $u_1+2u_2\geq 2n$, $u_1,u_2$ are even} \\
0 & \text{if $u_1+2u_2 < 2n$, $u_1,u_2$ are even} \\
-1 & \text{if $u_1+2u_2\geq 2n$, $u_1$ even, $u_2$ odd}\\
0 & \text{if $u_1+2u_2 < 2n$, $u_1$ even, $u_2$ odd}\\
0 & \text{if $u_1$ is odd}
\end{cases}
\end{align*}

\begin{lemma}
\lbl{lem.ineq10}
If  $\mu=u_1\lambda_1+u_2\lambda_2\in S_{n\lambda_1,2}$ then $u_1+2u_2\leq 2n$.
\end{lemma}
\begin{proof}
By Lemma \ref{lem.Plambda}, we have $\mu\in S_{n\l_1,2}\subset P_{2n\l_1}$. 
So by Inequality \eqref{ineq1} we have
$$
(2n\l_1-u_1\l_1-u_2\l_2,\l_2)\geq 0 \quad \text{i.e.,}\quad u_1+2u_2\leq 2n
$$
\end{proof}
From Corollary \ref{cor.nonzero} and Lemma \ref{lem.ineq10} we have

\begin{corollary}
$ c(u_1,u_2)\neq 0$ if and only if $u_1+2u_2=2n$. 
\end{corollary}
 
\begin{proof}[Proof of Theorem \ref{thm.T23}]
Set $s=\frac{u_1}{2}=n-u_2$, then $u_1^2+u_1u_2+u_2^2-n^2=3s^2$ and we have 
 
\begin{equation*}
\check{J}^{\mathfrak{g}}_{T(2,b),n\lambda_1}(q)=
\frac{\sum\limits_{s=0}^n (-1)^{s}q^{\frac{b}{2}s^2+(\frac{b}{2}-1)s} 
(1-q^{2s+1})(1-q^{n-s+1})(1-q^{n+s+2})}{(1-q)(1-q^{n+1})(1-q^{n+2})}
\end{equation*}
Replacing $q^n$ by $x$ and using Lemma \ref{stability.lemma1} it follows
that $(\hat{J}^{A_2}_{T(2,b),n\lambda_1}(q))$ is $c$-stable and its tail 
$G_b(x,q)$ is given by
\begin{align*}
G_b(x,q)&=\frac{\sum\limits_{s= 0}^\infty (-1)^sq^{\frac{b}{2}s^2+(\frac{b}{2}-1)s} 
(1-q^{2s+1})(1-xq^{1-s})(1-xq^{s+2})}{(1-q)(1-qx)(1-q^2x)} \\
&=\frac{\sum\limits_{s=0}^\infty(-1)^s((q^{\frac{b}{2}s^2+(\frac{b}{2}-1)s}
-q^{\frac{b}{2}s^2+(\frac{b}{2}+1)s+1})(1+q^3x^2)
+(q^{\frac{b}{2}s^2+(\frac{b}{2}+2)s+3}
-q^{\frac{b}{2}s^2+(\frac{b}{2}-2)s+1})x)}{(1-q)(1-qx)(1-q^2x)}
\end{align*}
Using $s=t+1$, we have
\begin{align*}
\sum\limits_{s=0}^\infty(-1)^{s+1}q^{\frac{b}{2}s^2+(\frac{b}{2}+1)s+1}
&=\sum\limits_{t=-1}^\infty (-1)^{-t} q^{\frac{b}{2}(t+1)^2-(\frac{b}{2}+1)(t+1)+1} 
=\sum\limits_{s\leq -1}(-1)^tq^{\frac{b}{2}s^2+(\frac{b}{2}-1)s}
\end{align*}
Therefore,
\begin{align*}
\sum\limits_{s=0}^\infty(-1)^s(q^{\frac{b}{2}s^2+(\frac{b}{2}-1)s}
-q^{\frac{b}{2}s^2+(\frac{b}{2}+1)s+1})
&=\sum\limits_{s=-\infty}^\infty (-1)^sq^{\frac{b}{2}s^2+(\frac{b}{2}-1)s}
=\theta_{b,\frac{b}{2}-1}(q)
\end{align*}
Similarly,
\begin{align*}
\sum\limits_{s= 0}^\infty q^{\frac{b}{2}s^2+(\frac{b}{2}+2)s+3}
-q^{\frac{b}{2}s^2+(\frac{b}{2}-2)s+1}
&=\sum\limits_{s=-\infty}^\infty (-1)^sq^{\frac{b}{2}s^2+(\frac{b}{2}+2)s+3}
=q^3\theta_{b,\frac{b}{2}+2}(q)
\end{align*}
Thus,
\begin{equation*}
G_b(x,q)=\frac{ \theta_{b,\frac{b}{2}-1}(q)(1+q^3x^2)+q^3
\theta_{b,\frac{b}{2}+2}(q)x }{(1-q)(1-qx)(1-q^2x)}
\end{equation*}
Note that by replacing $s$ with $s+1$ or $s$ by $-s$ in Equation 
\eqref{eq.theta} it follows that
\begin{equation*}
\theta_{b,c}(q)=-q^{\frac{b}{2}+c}\theta_{b,b+c}(q), \qquad \theta_{b,-c}(q)=
\theta_{b,c}(q) 
\end{equation*}
To compute $G_3(x,q)$, use $b=3,c=\frac{1}{2}$ in the above equation
and {\em Euler's Pentagonal Theorem} (discussed in detail in \cite{Bell})
to obtain that
\begin{align*}
q^2\theta_{3,\frac{7}{2}}(q)&=-\theta_{3,\frac{1}{2}}(q)=-(q)_\infty
\end{align*}
This completes the proof of Theorem \ref{thm.T23}.
\end{proof}

\subsection{The tail for $A_2$ and the $T(4,5)$ torus knot}
\lbl{sub.A2.T45}

In this section we compute the tail for the $c$-stable sequence
$(J^{A_2}_{T(4,b),n\rho}(q))$ for $b>4$ odd. This example shows that $c$-stability 
is a necessary notion for Conjecture \ref{conj.1}. 

Let
\begin{align*}
A_{b,0}(q) &=\sum\limits_{(s,t)}\sum\limits_{(u_1,u_2)} 
\e_{s,t} c_{s,t}(u_1,u_2) q^{\frac{b}{12}(u_1^2+u_1u_2+u_2^2)+(\frac{b}{4}-1)(u_1+u_2)}
(1-q^{u_1+1})(1-q^{u_2+1})(1-q^{u_1+u_2+2})
\\
A_{b,1}(q) &=
  \sum\limits_{(s,t)}\sum\limits_{(u_1,u_2)} 
\e_{s,t} q^{\frac{b}{12}(u_1^2+u_1u_2+u_2^2)+(\frac{b}{4}-1)(u_1+u_2)}
(1-q^{u_1+1})(1-q^{u_2+1})(1-q^{u_1+u_2+2})
\end{align*}
where the $(s,t)$ summation is over the set 
\be
\lbl{sign.tab1}
\begin{array}{|c|cccccc|}
 \hline
(s,t) & (0,0) & (2,-1) & (-1,2) & (0,3) & (3,0) & (2,2)
\\ \hline 
\e_{s,t} & 1 & -1 & -1 & 1 & 1 & -1
\\ \hline 
\end{array}
\ee
and $(u_1,u_2) \in \BN^2$ satisfies $u_1\equiv -s \bmod 4$, 
$u_1-u_2\equiv t-s \bmod 12$ and
\begin{align*}
c_{s,t}(u_1,u_2)=& 
\begin{cases}
1-\frac{2u_1+u_2+2s+t}{12}& \text{if} \, u_1+s\geq u_2+t \\
1-\frac{u_1+2u_2+s+2t}{12}& \text{if} \, u_1+s\leq u_2+t
\end{cases}
\end{align*}

\begin{proposition} 
\lbl{prop.A2.T45}
The tail of the $c$-stable sequence $(\hat{J}^{A_2}_{T(4,b),n\rho}(q))$ 
is given by
$$
\frac{1}{(1-xq)^2(1-x^2q^2)}
\left( A_{b,0}(q) + n A_{b,1}(q) \right)
$$
\end{proposition}

\begin{proof}
We will use Theorem \ref{thm.stable.limit} and unravel its notation. 
To begin with, for $a=4$, we have 
\begin{align*}
\calL_{n\rho,4}= 
\bigcup\limits_{\s\in W}4n\rho +\sigma(\rho)-\rho + 4\Lambda_r
= \bigcup\limits_{\s\in W} \sigma(\rho)-\rho+4\Lambda_r 
= \bigcup\limits_{\s\in W}\{\mu\in\Lambda| \mu+\rho-\sigma(\rho)\in 4\Lambda_r\}
\end{align*}
Since $\rho=\a_1+\a_2\in \Lambda_r$, we have 
$\calL_{n\rho,4}=\calL_{\rho,4}$ for all natural numbers $n$.
Let $\mu=u_1\l_1+u_2\l_2$ and $\rho-\sigma(\rho)=s\l_1+t\l_2$ where 
$(s,t)$ are given in \eqref{sign.tab1} and $(-1)^\s=\e_{s,t}$ as in 
\eqref{sign.tab1}. In weight coordinates we have  
\begin{equation}
\lbl{lat.4}
\calL_{\rho,4} = \bigcup\limits_{(s,t)}\{(u_1,u_2)\in\mathbb{Z}^2:
u_1\equiv -s \bmod 4, u_1-u_2\equiv t-s \bmod 12\}
\end{equation}
Next we compute the plethysm multiplicities. Equation \eqref{eq.mula} 
implies that
\begin{align*}
m^{\mu}_{n\rho,4}&= \sum\limits_{\sigma\in\W}(-1)^\sigma 
m_{n\rho}^{\frac{\mu+\rho-\sigma(\rho)}{4}}\\
&= m_{n\rho}^{\frac{\mu}{4}}-m_{n\rho}^{\frac{\mu+2\l_1-\l_2}{4}}
-m_{n\rho}^{\frac{\mu-\l_1+2\l_2}{4}}+m_{n\rho}^{\frac{\mu+3\l_1}{4}}
+m_{n\rho}^{\frac{\mu+3\l_2}{4}}-m_{n\rho}^{\frac{\mu+2\l_1+2\l_2}{4}}
\end{align*}
Since $n\rho\in\Lambda_r$, $m_{n\rho}^{\nu}\neq 0$ only if 
$\nu\in\Lambda_r$. Therefore at most one of the terms in the above
equation is non-zero. Equation \eqref{eq.komult} gives
\begin{align*}
m_{n\rho}^{\mu}=
\begin{cases}
1+\frac{2m_1+m_2}{3}-\frac{2u_1+u_2}{3} & \text{if} \quad u_1\geq u_2\\
1+\frac{m_1+2m_2}{3}-\frac{u_1+2u_2}{3} & \text{if} \quad u_1\leq u_2
\end{cases}
\end{align*}
Therefore 
\begin{align*}
m^{\mu}_{n\rho,4}=\e_{s,t}
\begin{cases}
1+n-\frac{2u_1+u_2+2s+t}{12} &\text{if} \,\, u_1\equiv -s
\bmod 4, u_1-u_2\equiv t-s \bmod 12 \,, u_1+s\geq u_2+t\\
1+n-\frac{u_1+2u_2+s+2t}{12} &\text{if} \,\, u_1\equiv -s
\bmod 4, u_1-u_2\equiv t-s \bmod 12 \,,  u_1+s\leq u_2+t
\end{cases}
\end{align*}
where $\e_{s,t}$ is given from \eqref{sign.tab1}. 
Since $\mu_{n\rho,4}=0$, we have $\hat{\calL}_{n\rho,4}
=\calL_{n\rho,4},\, \hat{P}_{n\rho}=P_{n\rho},\, \hat{S}_{n\rho,4}=S_{n\rho,4}$. 
Theorem \ref{thm.stable.limit} concludes the proof of
Proposition \ref{prop.A2.T45}.
\end{proof}

\begin{exercise}
\lbl{ex.A1b}
Show that 
\be
\lbl{eq.Ab1.theta}
A_{b,1}(q)=\sum_{m_1,m_2 \in \BZ} q^{4 b(m_1^2 + 3 m_1 m_2 + 3 m_2^2)
+ (b -4)(2 m_1 + 3 m_2)}
(1 - q^{4 m_1 +1}) (1 - q^{4 m_1 + 12 m_2 +1}) (1 - q^{8 m_1 + 12 m_2+2})
\ee
The above equation shows that $A_{b,0}(q)$ is a sum of theta series of rank 
$2$, hence a modular form of weight $1$; see \cite{123}.
\end{exercise}

\subsection*{Acknowledgment}
The authors wish to thank Jon Stembridge for useful communications.

%

\bibliographystyle{hamsalpha}
\bibliography{biblio}
\end{document}